\documentclass[journal]{IEEEtran}
\IEEEoverridecommandlockouts

\usepackage{grffile}
\usepackage{amsmath,amssymb,amsbsy}

\usepackage{amsthm} 
\usepackage{color}
\usepackage{xcolor}
\usepackage{comment}
\usepackage{lipsum}
\usepackage{graphicx}
\usepackage[normalem]{ulem}
\usepackage{enumerate}
\usepackage{epstopdf}
\usepackage{bbm}
\usepackage{subcaption}
\usepackage[normalem]{ulem}
\usepackage{enumerate}
\usepackage{epstopdf}
\usepackage{bbm}
\usepackage{tikz}
\usepackage{comment}

\usepackage{multicol}
\usepackage{todonotes}
\usepackage{hyperref}

\newcommand{\x}{x}

\newcommand{\Rd}{\mathbb{R}^d}
\newcommand{\R}{\mathbb{R}}

\newcommand{\map}[3]{#1: #2 \rightarrow #3}
\newcommand{\rev}[1]{{\color{black} #1}}

\newtheorem{theorem}{Theorem}[section]
\newtheorem{lemma}[theorem]{Lemma}
\newtheorem{definition}[theorem]{Definition}	
\newtheorem{proposition}[theorem]{Proposition}
\newtheorem{corollary}[theorem]{Corollary}	
	
\newtheorem{problem}[theorem]{Problem}

\newtheorem{assumption}[theorem]{Assumption}

\newcommand{\B}{\mathcal{B}}

\newcommand{\D}{\mathcal{D}}
\newcommand{\T}{\mathcal{T}}

\begin{document}
	
	\title{Score Matching Diffusion Based Feedback Control and Planning of Nonlinear Systems}

    \author{Karthik Elamvazhuthi, Darshan Gadginmath, and Fabio Pasqualetti
\thanks{This material is based upon work supported in part by award AFOSR-
FA9550-20-1-0140.}
\thanks{Karthik Elamvazhuthi is with the T-5 Applied Mathematics and Plasma Physics Group, Los Alamos National Laboratory (e-mail: karthikevaz@lanl.gov). }
\thanks{Darshan Gadginmath is with the Department of Mechanical Engineering, University of California, Riverside, CA, USA (e-mail: dgadg001@ucr.edu).}
\thanks{Fabio Pasqualetti is with the 
Electrical Engineering and Computer Science, University of California, Irvine, CA, USA ( (e-mail: fabiopas@uci.edu).}}

	\maketitle

	\begin{abstract} 
		
      \rev{In this paper, we propose a deterministic diffusion-based framework for controlling the probability density of nonlinear control-affine systems, with theoretical guarantees for drift-free and linear time-invariant (LTI) dynamics. The central idea is to first excite the system with white noise so that a forward diffusion process explores the reachable regions of state space, and then to design a deterministic feedback law that acts as a denoising mechanism driving the system back toward a desired target distribution supported on the target set. This denoising phase provides a feedback controller that steers the control system to the target set. In this framework, control synthesis reduces to constructing a deterministic reverse process that reproduces the desired evolution of state densities. We derive existence conditions ensuring such deterministic realizations of time-reversals for controllable drift-free and LTI systems, and show that the resulting feedback laws provide a tractable alternative to nonlinear control by viewing density control as a relaxation of controlling a system to target sets. Numerical studies on a unicycle model with obstacles, a five-dimensional driftless system, and a four-dimensional LTI system demonstrate reliable diffusion-inspired density control.}
	\end{abstract}
	
	
	\section{Introduction
	}
  
	\rev{Feedback control of nonlinear systems remains a central challenge in control theory. Unlike linear systems, which admit systematic stabilization via LQR, pole placement, and Lyapunov methods, nonlinear systems lack a unified framework due to obstructions such as non-convex optimal control formulations~\cite{clarke2013functional} and topological constraints on smooth feedback~\cite{brockett1983asymptotic,khalil2002nonlinear}. Although feedback linearization~\cite{krener1999feedback}, Lyapunov-based control~\cite{freeman1996control}, and model predictive control~\cite{grune2017nonlinear} address specific classes of systems, the absence of a broadly applicable synthesis method motivates alternative problem formulations.}

    \rev{Recent advances in \emph{generative modeling} provide a complementary viewpoint for constructing structured behaviors for systems. In these models, random perturbations are first introduced to explore the space of possible states, and a subsequent ``denoising'' phase reconstructs coherent samples from this noisy representation. 
This diffusion--denoising cycle underlies modern score-based and diffusion models \cite{ho2020denoising,song2020score,}, which learn to reverse a noise-driven diffusion process to recover structured data distributions. 
The same principle of exploration through noise followed by structured correction offers a useful analogy for control synthesis, where excitation of the system dynamics and subsequent stabilization play analogous roles.


This paper adopts this perspective to formulate a new approach to feedback control for nonlinear systems. 
We interpret control synthesis as a two-phase process: 
a \emph{diffusion phase}, in which the system is excited with noise to explore its reachable state space, and a \emph{deterministic denoising phase}, in which feedback drives the system toward a desired target distribution. 
The denoising phase is posed as a finite-time \emph{density-control problem}, where the feedback law realizes the time-reversal of the forward diffusion and steers the system from the noise-excited distribution back to the initial distribution of the forward phase.  By embedding feedback within this diffusion--denoising cycle, control design reduces to constructing a deterministic reverse process that reproduces the desired evolution of the system’s state densities.

	\subsection{Contribution}
    Our main contributions are as follows:
	
\begin{enumerate}
    \item \textbf{Diffusion--denoising control algorithms:} 
    We develop two algorithms that synthesize feedback laws by reversing a forward diffusion process under the system dynamics (Section~\ref{sec:Algos}).
    Algorithm~1 minimizes the KL divergence between the controlled and reference densities in the style of denoising diffusion probabilistic models, while Algorithm~2 learns a nonholonomic score function that directly approximates the time-reversing feedback law.

    \item \textbf{Existence and realizability theory:} 
    We derive rigorous conditions under which a deterministic feedback law can exactly reproduce the time-reversed evolution of a diffusion process.  
    These results establish the existence and well-posedness of deterministic reverse processes for two system classes:  
    (\emph{i}) controllable drift-free nonlinear systems satisfying the Chow--Rashevsky condition (Theorem \ref{thm:detrel} and Theorem \ref{thm:detcontaffalg2}), and  
    (\emph{ii}) controllable and asymptotically stable linear time-invariant (LTI) systems (Theorem \ref{thm:revglor2}). These results yield target set convergence guarantees (Corollary \ref{cor:conafftargset1}, Corollary \ref{cor:conafftargset2} and Corollary \ref{cor:lintargset}).


    \item \textbf{Numerical validation:} 
    \rev{We demonstrate both algorithms on a unicycle model with obstacles, a five-dimensional driftless system, and a four-dimensional LTI system, illustrating finite-time density steering and stabilization (Section~\ref{sec:numer}).
}
\end{enumerate}
	}
	
	\rev{\paragraph*{Comparison with prior work on density control.}

The control of probability densities has been studied extensively within the frameworks of optimal transport and stochastic control, where system behavior is described in terms of evolving distributions rather than individual trajectories. Foundational work extended optimal transport theory to nonlinear control-affine systems, including settings with nonholonomic constraints~\cite{agrachev2009optimal,figalli2010mass}. Subsequent generalizations relaxed assumptions on the system dynamics and cost functions~\cite{elamvazhuthi2023dynamical,elamvazhuthi2024benamou}, and extensions to feedback-linearizable systems were proposed in~\cite{caluya2020finite}. For linear systems, finite-horizon covariance and density steering problems have been studied primarily for Gaussian distributions~\cite{bakolas2018finite,okamoto2018optimal}. While these approaches offer rigorous optimality guarantees, they typically require solving high-dimensional partial differential equations or the optimization problem might be non-convex.  In contrast, the framework proposed here does not pose density steering as an optimal control problem over the space of probability densities. Instead, it constructs an auxiliary forward diffusion solely to generate a reference density trajectory and then synthesizes a deterministic feedback law that realizes the time-reversed density evolution under the Liouville equation. 

}

\rev{\paragraph*{Concurrent work on time-reversal based control.}
Two concurrent and independent studies also explore time-reversal ideas for control:
Grong \emph{et al.}~\cite{grong2024score} developed geometric score-matching for diffusion bridges on sub-Riemannian manifolds, and Mei \emph{et al.}~\cite{mei2025time} used time-reversal theory to design score-based feedback laws for stochastic control-affine systems.
Both retain stochastic noise in the reverse process and address only point-to-point steering (Dirac targets).
Our framework differs in two respects: the reverse process is deterministic, and it steers between general probability densities.
We also prove existence and realizability of feedback control for drift-free and LTI systems.}

\section{Problem setting and preliminary notions}
\label{sec:problem_formulation}
	Consider the control affine, nonlinear control system
	\begin{align}
		\label{eq:ctrsys}
		\dot x = g(x,u) = g_0(x)+\sum_{i=1}^m  g_i(x) u_i,  
	\end{align}
	where $x \in \mathbb{R}^d$ denotes the state  of the system, $u_i$ the $i$-th control input, and $g_i \in C^{\infty}(\mathbb{R}^d;\mathbb{R}^d)$ are smooth vector fields. The model \eqref{eq:ctrsys} is commonly used in robotics to describe the dynamics of ground vehicles, underwater robots, manipulators, and other nonholonomic systems~\cite{JCL:2012}. 
The goal of this paper is to design feedback laws that steer the state of such systems toward a desired region or distribution within a finite horizon.
    

    \rev{\begin{problem}[\rev{Target-set control in finite time}]
\label{prob:set_convergence}

Given the dynamics~\eqref{eq:ctrsys}, an initial probability of initial states $p_0$, a target
set $\Omega_{\rm target} \subset\mathbb{R}^d$, and a finite horizon $T>0$, design a measurable
feedback law $u=\pi(t,x)$ such that $x(T) \in \Omega_{\rm target}$ with probability 1.
\end{problem}}

	\begin{figure*}
		\centering
		\includegraphics[width=0.94\textwidth]{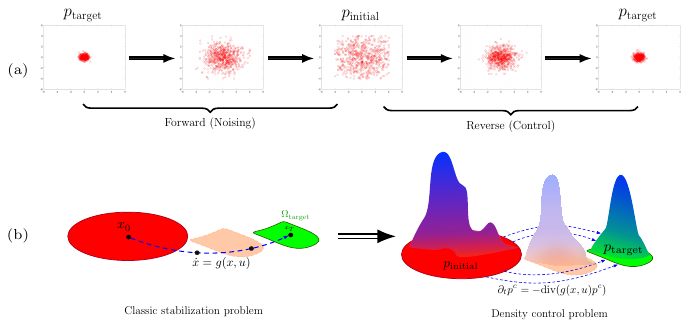}
		\caption{(a) DDPMs in two dimensions, transforming the data distribution $p_{\text{target}}$ to $p_n$ and learning to reverse the process. (b) Reformulation from the classical control problem to the density control problem. }\label{fig:diffusion-models2d}
	\end{figure*}

    \rev{A natural way to address the above problem would be to control the marginal density of the state $p^{\text{c}}(t)$ for all $t \in [0,T]$ along the dynamics. Here the marginal density $p^c_t = \mathbb{P}(x(t) \in dx)$.}  Therefore, in order to address the above problem, we consider a more general state feedback control problem, relating it to probability density control. \rev{Toward this end, we will need a description of how the density $p_c(t)$ evolves under the dynamics \eqref{eq:ctrsys} with feedback $u = \pi(t,x)$. This is given by the Liouville equation or the continuity equation \cite{santambrogio2015optimal},}
	\begin{align}
		\label{eq:fwdpdf_cs}
		\partial_t p^{\textup{c}}_t = -\textup{div}(g(x,\pi(t,x)) p^{\textup{c}}_t), 
	\end{align}
	for a given initial density $p^{\textup{c}}_0 = p_0$. This results in a more general problem as follows.
    \rev{
	\begin{problem}[\rev{Target density control in finite time}]
\label{prob:target_measure}
Given the system~\eqref{eq:ctrsys} and its density evolution~\eqref{eq:fwdpdf_cs},
an initial density $p_0$ and a finite horizon $T>0$, 
design a feedback law $u=\pi(t,x)$ such that the controlled density $p_c$ satisfies
\begin{equation}
\label{eq:target_density}
    p^\mathrm{c}_T = p_{\rm target},
\end{equation}
where $p_{\rm target}$ is a desired target density with support $\Omega_{\rm target}$.
\end{problem}
}
Note that formally the above problem addresses \ref{prob:set_convergence} if we choose $p_{\rm target} = \mathbf{1}_{\Omega_{\rm target}}$. 	This framework unifies classical and probabilistic control, providing an alternative approach to nonlinear control. Figure~\ref{fig:diffusion-models2d} illustrates this reformulation.

    \subsection*{Density Control using Denoising Diffusion}
  
A standard way to approach Problem~\ref{prob:target_measure} is to formulate it as an optimal control or optimal transport problem. 
However, such formulations often suffer from the curse of dimensionality, since the Liouville equation imposes an infinite-dimensional constraint on the density evolution. 
To address this issue, we consider an alternative viewpoint based on tracking a reference stochastic process, inspired by Denoising Diffusion Probabilistic Models (DDPMs). See \ref{sec:bgrnd} for a brief introduction to DDPMs.

Suppose we define a {\it forward stochastic process} $X^{\mathrm{f}}(t)$ such that its initial state $X^{\mathrm{f}}(0)$ is distributed according to the desired terminal density $p_{\rm target}$, 
and its marginal densities {\color{black} $p_t^{\mathrm{f}} = \mathbb{P}(X^{\mathrm{f}}(t) \in dx)$} {\color{black} asymptotically converge to a \emph{noise distribution} $p_n$ that is easy to sample from, as $t$ increases. Examples of possible choices of noise distributions are Gaussians and the uniform distribution.}
Intuitively, this process represents a diffusion that gradually destroys the structure in $p_{\rm target}$, transforming it into $p_n$. The density--control problem can then be recast as a \emph{density--tracking} problem: 
design a feedback law $u=\pi(t,x)$ such that the solution $p_t^{c}$ of the Liouville equation~\eqref{eq:fwdpdf_cs} satisfies
\begin{equation}
    p_t^{\mathrm{c}} = p_{T-t}^{\mathrm{f}}, \qquad t\in[0,T],
\end{equation}
with the initial condition $p_0^{c}=p_n$. 
In other words, the controlled system should deterministically reproduce the time--reversed evolution of the forward diffusion, driving the density from $p_n$ back to $p_{\rm target}$.

This formulation transforms the density--control problem from a high-dimensional constrained optimization into a regression problem as shown by the literature in score-matching diffusion models~\cite{song2020score}. The objective is to now learn or approximate a drift field that realizes the reverse density evolution which is computationally scalable and can be computed from data. We need to design feedback control $u(t,x)$ such that that the drift field in the reverse process is the dynamics~\eqref{eq:ctrsys} under the control $u(t,x)$ to solve Problem~\ref{prob:target_measure}.
    
	Let $\Omega$ be a compact domain in which the control system is confined. We will consider two instances of the denoising diffusion based control problem for different forward processes. 
	
	\rev{
	\begin{enumerate}
		\item \textbf{Algorithm 1 (Generic forward process):} The forward process $X^\mathrm{f}$ follows a stochastic differential equation (SDE) of the form \footnote{The {\color{black} scaling constant} $\sqrt{2}$ simplifies Equation~\eqref{eq:fwdpdf}.}:
        
		\begin{align}
        \label{eq:sde}
			\textup{d} X^\text{f} = V(X^\text{f})\textup{d}t + \sqrt{2} \, \textup{d} W + {\color{black} n(X^\mathrm{f}) \textup{d} Z},
		\end{align}
        \rev{In this choice of forward process, we use a generic auxiliary diffusion that transfers the state density to the noise distribution and does not inherit any structural information from the control system \eqref{eq:ctrsys}.}

		\item \textbf{Algorithm 2 (Forward process under system~\eqref{eq:ctrsys}):} The forward process evolves according to an SDE that inherits the structure of \eqref{eq:ctrsys}:
\begin{align}
	\label{eq:nhmsde}
	dX^\mathrm{f} = &-g_0(X^\mathrm{f})\textup{d}t+\sum_{i=1}^m v_i(X^\mathrm{f})g_i(X^\mathrm{f})\textup{d}t \\ \nonumber 
	&+  \sqrt{2}\sum_{i=1}^m g_i(X^\mathrm{f}) \odot \textup{d}W + n(X^\mathrm{f})\textup{d}Z .
\end{align}

\rev{Here, the functions $v_i:\mathbb{R}^d\to\mathbb{R}$ are \emph{noise-shaping feedback coefficients} used in the forward diffusion to ensure that the forward density converges to the prescribed noise distribution $p_n$, given the actuation constraint imposed by control channels in \eqref{eq:ctrsys}.  }

	\end{enumerate}
	}
	Here, $\odot$ emphasizes the fact that we use the Stratonovich integral. Additionally,  $Z(t)$ is the {\it boundary local time}, which acts only when the state reaches the domain $\partial \Omega$ to keep the trajectory in the prescribed domain \cite{pilipenko2014introduction}. Note that the functions $v_i(x)$ play the same role in Algorithm 2, as the function $V(x)$ in Algorithm 1. They should not be confused with the policy $\pi(t,x)$.
    
	From a technical point of view, the denoising--based control design described above raises two fundamental questions. 
First, under what conditions is such a time reversal of the density dynamics possible? 
Second, if the reversal exists, does there also exist a probability measure that conclusively satisfies Problem~\ref{prob:set_convergence}? 
To formalize these questions, we introduce the notion of a \emph{deterministic realization}.

\begin{definition}[Deterministic realization]
We say that there exists a \textbf{deterministic realization} associated with the control policy $\pi(t,x)$ 
if there exists a probability measure $\mathbb{P}$ on $C([0,T];\mathbb{R}^d)$ such that, for $\mathbb{P}$--almost every 
$\omega\in C([0,T];\mathbb{R}^d)$, the corresponding trajectory $\gamma_t$ satisfies
\begin{equation}
    \dot \gamma_t = g_0(\gamma_t) + \sum_{i=1}^m g_i(\gamma_t)\,\pi_i(t,\gamma_t),
\end{equation}
and its marginal distributions satisfy
\begin{equation}
    \mathbb{P}(\gamma_t \in dx) = p^{\mathrm{f}}_{T-t}(x), \qquad \forall\,t\in[0,T].
\end{equation}
\label{def:deterministic_realization}
\end{definition}

In other words, a deterministic realization ensures that the controlled trajectories generated by $\pi(t,x)$ 
produce a probability flow that exactly reproduces the time-reversed densities of the forward process, \rev{where the reverse process is \eqref{eq:ctrsys} for a suitable choice of time-reversing control law $\pi(t,x)$}. 
Given this notion, we can now pose the main question that will be investigated for each of the two algorithms presented in the following sections.

\begin{problem}[Existence of time-reversing control law and deterministic realization]
\label{prb3}
\rev{Given $p^\mathrm{f}$ is the distribution of the forward process of either Algorithm 1 or 2, does there exist a feedback control policy $u=\pi(t,x)$ such that the density of the controlled state \eqref{eq:ctrsys}} satisfies
\begin{equation}
    p^{\mathrm{c}}_t = p^{\mathrm{f}}_{T-t}, \qquad \forall\,t\in[0,T]?
\end{equation}
If such a control policy exists, does there also exist a deterministic realization of the corresponding reverse process in the sense of Definition~\ref{def:deterministic_realization}?
\label{problem:realization}
\end{problem}

Classical results on time reversal of diffusion processes \cite{anderson1982reverse,haussmann1986time,cattiaux1988time} establish conditions under which the reverse process of a stochastic system can be realized. A significant challenge addressed by these works is that the coefficients of the time--reversed dynamics are not globally Lipschitz, 
\rev{since the resulting control laws contain terms of the form $\nabla \log p^f_t$ and $\nabla \log p^f_0$ might not be Lipschitz.} Hence, one cannot use classical existence-uniqueness criteria to construct the reverse process over the time horizon. We seek to address a similar challenge of time reversal in our work. However, our work differs in two fundamental ways:
	
	\begin{itemize}
		\item The reverse process in our setting is {\it deterministic}, which is important for control applications where injecting noise into the system might be undesirable. 
		\item Unlike prior work that assumes the noise and control channels are the same, we relax this assumption for systems without drift.
	\end{itemize}

	\section{Algorithms for DDPM-based control }
	\label{sec:Algos}

In this section, we provide details on the two denoising diffusion based algorithms that we will use to solve the probability density control problem. The first algorithm is expected to {\color{black} approximately achieve the control objective better through the choice of the loss function}, while the second algorithm is computationally more scalable  as it transforms the reversing problem into a regression problem.

\subsection{Generic forward process - Algorithm 1}
In the SDE for the forward process~\eqref{eq:sde}, if we set $V(x) = \nabla \log p_n$, then the forward process provides a reference trajectory $p^{\textup{f}}_{T-t}$ in the set of probability densities such that $p^{\textup{f}}_T \approx p_n$ and $p_0 = p_{\text{target}}$. This follows from theory of diffusion processes and their semigroups \cite{bakry2013analysis}. Thus if there exists a controller $\pi(t,x)$ that can ensure system \eqref{eq:ctrsys} exactly tracks this density trajectory, we will have that $p_T^{\textup{c}} \approx p_{\text{target}}$. 

Such convergence is guaranteed by the Fokker Planck equation that governs the evolution of the density $p$ of the state  $X^f$ of~\eqref{eq:sde}: 
	\begin{align}
		\label{eq:fwdpdf}
		\partial_t p^\mathrm{f} = \Delta p^\mathrm{f} - \text{div}(V(x) p^\mathrm{f}) = \text{div}  ([\nabla \log p^\mathrm{f} -V(x)] p^\mathrm{f}) . 
	\end{align}
	where  $\Delta : = \sum_{i=1}^d \partial_{x^2_i}$ and $\text{div} : = \sum_{i=1}^d \partial_{x_i}$ denote the Laplacian and the  divergence operator, and $p_0 = p_\text{target}$.
	
	
	When $\Omega$ is a strict subset of $\mathbb{R}^d$, this equation is additionally supplemented by a boundary condition, known as the {\it zero flux} boundary condition
	\begin{equation}
		\vec{n}(x)\cdot ( \nabla p^f_t(x) - V(x)) = 0 ~~~ \mbox{on} ~~~ \partial \Omega ,
	\end{equation}
	where $\vec{n}(x)$ is the unit vector normal to the boundary $\partial \Omega$ of the domain $\Omega$.
	This boundary condition ensures that $\int_{\mathbb{R}^d} p_t(x)dx = 1$ for all $t\geq 0$. An advantage of considering the situation of bounded domain is that one can choose $V\equiv 0$ and the noise density can be taken to be the uniform density on $\Omega$. This property will be useful when solving our control problem.


\rev{To identify the controller $\pi(t, x)$, we seek to solve the following minimization problem:}
\begin{align}
	\min_{\rev{\pi:= NN(t,x,\theta)}} \frac{1}{T}\int\limits_0^T \text{KL}(p^{\textup{c}} \big|p^{\textup{f}}_{T-t}) \ \textup{d}t. \label{eqn:cost-continuous}
\end{align} 
subject to the dynamics \begin{align}\partial_t p^{\textup{c}} = -\text{div} (g(x,\texttt{NN}(t,x,\theta)) p^{\textup{c}}),
	\label{eqn:NNct}
\end{align} 
where $p^{\textup{c}}_0 = p_n$. 
Here $\text{KL}(p^{\textup{c}}\big|p^{\textup{f}})$ denotes the KL divergence between the density of the control system and the forward density. The KL divergence between any two densities $Q$ and $R$ defined on a set $\Omega$ is given by
\begin{align}
	\text{KL}(Q\big|R) &= \int\limits_{x \in \Omega} Q(x) \log\left(\frac{Q(x)}{R(x)} \right) ~\textup{d}x.
\end{align}
\rev{In practice, we need to estimate $\text{KL}(Q\big|R)$ using samples of $Q$ and $R$. For this, we use a kernel density estimation to approximate $KL(Q|R)$ using samples of $Q$ and $R$. See \cite{carrillo2019blob}.}
\subsection{Forward process under system~\eqref{eq:ctrsys} - Algorithm 2}

A drawback of the previous algorithm is that \eqref{eqn:cost-continuous}-\eqref{eqn:NNct} is a constrained optimization problem, with constraints \eqref{eqn:NNct}. This implies that the equation \eqref{eqn:NNct} needs to be solved in every pass of the gradient descent. To address this issue, we present an alternative algorithm.

Toward this end, we associate with each vector-field $g_i:\mathbb{R}^d \rightarrow \mathbb{R}^d$
a first order differential operator $Y_i$, given by
\begin{align*}
	Y_i ~ \cdot = \sum_{j=1}^d g^j_i(x) \partial_{x_j}
\end{align*}

where $Y_i^*$ is the formal adjoint of the operator $Y_i$ and is given by
\[Y_i^* = \sum_{j=1}^d -\frac{\partial}{\partial x_j} (g_i^j  ~~~\cdot) ~~~ \]

The operator $Y_i^*$ is the formal adjoint of $Y_i$ can be seen from the computation

\[ \int_{\Rd} g(x) Y_if(x)dx = \int_{\Rd} f(x)Y^*_ig(x) dx \]
for every smooth compactly supported function $f,g \in C^{\infty}_c(\R^d)$, by applying integration by parts.

Using these operators, the continuity equation \eqref{eq:fwdpdf_cs} can be expressed as
\begin{equation}
	\label{eq:kolmog2}
	\partial_t p^f  = Y^*_0p^c + \sum_{i=1}^m  Y_i^*(\pi_i(t,x) p^c)  
\end{equation}
\textbf{Forward process under system~\eqref{eq:ctrsys}.}
While in the previous algorithm, we retained \eqref{eq:sde} as the forward process, in this section we take \eqref{eq:nhmsde} as the forward process. \rev{By converting the Stratonovich SDE \eqref{eq:nhmsde} to its equivalent Itô form and applying the standard infinitesimal generator formula for diffusion processes \cite[Theorem 7.3.3]{oksendal2013stochastic}}, the  distribution $p^f$ of the process \ref{eq:nhmsde} evolves according to 

\begin{align}
	\partial_t p^f = -Y^*_0p^f + \sum_{i=1}^m  (Y_i^*)^2 p^f  +\sum_{i=1}^m  Y_i^*(v_i(x) p^f)  
	\label{eq:kfe}
\end{align}
First, an important requirement for the forward process is that it has the equilibrium distribution $p_{n}$. For this purpose, we need to control the long term behavior of the probability distribution of the forward process \eqref{eq:nhmsde}.

We will assume that we can construct control laws  $v_i(x)$, such that above PDE can be transformed to the form,
\begin{equation}
	\partial_t p^f  = -Y^*_0p^f -\sum_{i=1}^m  Y_i^*Y_i p^f +\sum_{i=1}^m  Y_i^* ( w_i(x)p^f)
\end{equation}
for some choice of $w_i(x)$, so that $p_n$ is it's equilibrium solution.
We will show in Section \ref{sec:results} that for each $p_n$ such a choice of $v_i(x)$ and $w_i(x)$ always exists, whenever the system is driftless. When the system is a linear time invariant system, one can choose $p_n$ from a class of Gaussian distributions by setting $w_i(x) \equiv 0$. 

We can rewrite this equation in the form
\begin{align}
	\label{eq:transf}
	\partial_t p^f  &=  -Y_0p^f-\sum_{i=1}^m  Y_i^*\left(\dfrac{Y_i p^f}{p^f}\right) p^f +\sum_{i=1}^m  Y_i^* ( w_i(x)p^f)  
\end{align}
Therefore, we have \textbf{Nonholonomic score loss}.
This gives us the nonholonomic score matching objective where we want to approximate the nonholonomic score $\frac{Y_ip^f}{p^f}$ using the neural network $s^{\theta}_i(t,x)$
\begin{eqnarray}
	\min_{\theta}\int_{0}^T \int_{\mathbb{R}^d}\|s^{\theta}(t,x) -  \nabla_c \log p^f_t(x) \|^2 p^f_{T-t}(x) dx
\end{eqnarray}
where $\nabla_c $ is \textbf{nonholonomic gradient operator} defined by 
\begin{align}
	\nabla_c  f & = [ Y_1f,..., Y_mf  ]^T  \nonumber \\
	&=[\sum_{j=1}^d g^j_1(x) \partial_{x_j }f(x),...,\sum_{j=1}^d g^j_m(x) \partial_{x_j } f(x)]^T
\end{align}
{\color{black} The objective above cannot be optimized in a sample-based manner since it involves the score term $\nabla_x \log p_t^f(x)$. Alternatively, applying integration by parts yields an equivalent score-matching loss that can be written as an expectation involving only the neural network $s^{\theta}(t,x)$ (and known quantities), eliminating the explicit dependence on $\nabla_x \log p_t^f(x)$. This idea is due to \cite{hyvarinen2005estimation}. However, this approach yielded degenerate results in our case. Therefore, we used a kernel density approximation to estimate the score from the samples.}

If we use the exact score $\nabla_c  \log p^f_{T-t}$ to execute the control $\pi(t,x) = \nabla_c  \log p^f_{T-t} -v(x)$ in the reverse process \eqref{eq:ctrsys} and $p^c \approx p^f_T  $, then substituting in equation we get that 
$p^c_t = p^f_{T-t}$ for all $ t\in [0,T]$ and the continuity equation \eqref{eq:kolmog2} becomes

\begin{align}
\partial_t p^c
&= Y_0^* p^c
 + \sum_{i=1}^m Y_i^*\!\left(\frac{Y_i p^f_{T-t}}{p^f_{T-t}}\,p^c\right)
 - \sum_{i=1}^m Y_i^*\!\left(w_i(x)\,p^c\right).
\label{eq:24_retyped}
\end{align}

\section{Analysis}
\label{sec:results}

The primary objective of this section is to establish a rigorous mathematical foundation for the deterministic realization of diffusion-based control processes (Problems \ref{prob:set_convergence}, \ref{prob:target_measure} and \ref{problem:realization}). 
\rev{In order to facilitate the analysis, we introduce a set of standard regularity assumptions ensuring that the associated PDEs are sufficiently non-degenerate. These include assumptions on controllability (Assumptions~\ref{asmp2} and~\ref{asmp:lti}), regularity of the domain (Assumption~\ref{asmp1}), and smoothness of the noise distribution (Assumption~\ref{asmp3}). These assumptions are analytical in nature and are imposed to guarantee well-posedness and invertibility properties of the operators used in the time-reversal analysis. They are not intended as additional design constraints that must be verified algorithmically.} 
\rev{
For each algorithm presented in the previous section, we will proceed with the analysis in with the following steps :
\begin{enumerate}
	\item \textbf{Step 1: Mathematical Preliminaries} – We introduce key concepts from geometric control theory and PDE theory, including weak solutions to the Fokker-Planck equation and the continuity equation.
    \item \textbf{Step 2: Convergence of Forward Process to Noise Distribution} - We ensure that the forward process will converge to the noise distribution. See Theorem \ref{thm:driftnoise} and Theorem \ref{thm:linstab}.
    \item \textbf{Step 3: Existence of Time Reversing Control} - Deriving regularity results of the corresponding Fokker-Planck equations, we establish conditions under which a deterministic feedback law can enforce time-reversed evolution of probability densities. See Lemma \ref{lem:extra}, Proposition \ref{prop:realdrift} and Proposition \ref{prop:linreadl}.
	\item \textbf{Step 4: Deterministic Realization of Reverse Process} -
    Using the superposition principle we establish the deterministic realization of the reverse process. See Theorem \ref{thm:detrel}, Theorem \ref{thm:detcontaffalg2}, and  
    Theorem \ref{thm:revglor2}.
	\item \textbf{Step 5: Target set convergence} – Using the deterministic realization result we prove that density control implies control to target sets in probability. See Corollary \ref{cor:conafftargset1}, Corollary \ref{cor:conafftargset2} and Corollary \ref{cor:lintargset}.
\end{enumerate}
}

\subsection{Analysis of Algorithm 1 (Driftless Systems)}

We address the well-posedness of problem \eqref{eqn:cost-continuous} from a theoretical point of view. The nonlinear dynamics of system~\eqref{eq:ctrsys} play a crucial role in the feasibility of this problem. In typical DDPM problems for generative modeling, noise can be added in all directions for both the forward and reverse processes. Therefore, score-matching techniques can be employed to learn the reverse process. However, for the feedback control problem in consideration, if the dynamical system is not fully actuated, that is $d=m$, then the feasibility of problem~\eqref{eqn:cost-continuous} is not clear. Here, we address the case when such a controller can be obtained.

Let $\mathcal{V} = \lbrace g_0, g_1,...,g_m \rbrace$, $m \leq d$, be a collection of smooth vector fields $g_i : \mathbb{R}^d \rightarrow \mathbb{R}^d$. Let $[f,g]$ denote the Lie bracket operation between two vector fields $f: \mathbb{R}^d \rightarrow \mathbb{R}^d$ and $g: \mathbb{R}^d \rightarrow \mathbb{R}^d$, given by
where $\partial_{i}$ denotes partial derivative with respect to coordinate $i$.
\begin{equation}
[f,g]_i = \sum_{j=1}^d f^j \partial_{x_j} g^i - g^j \partial_{x_j} f^i.
\end{equation}

We define $\mathcal{V}^0 =\mathcal{V}$. For each $i \in \mathbb{Z}_+$, we define in an iterative manner the set of vector fields $\mathcal{V}^i = \lbrace [g, h]; ~g \in \mathcal{V}, ~h \in \mathcal{V}^{j-1}, ~j=1,...,i \rbrace$.  We will assume that the collection of vector fields $\mathcal{V}$ satisfies following condition the {\it Chow-Rashevsky} condition \cite{agrachev2019comprehensive} (also known as {\it H\"{o}rmander's condition} \cite{bramanti2014invitation}).
For the purposes of this section, we will need the following assumption.

\begin{assumption}
\label{asmp2}
\textbf{(Controllable driftless system}) There is no drift in the system ($g \equiv 0$).
The Lie algebra generated by the vector fields $\mathcal{V}$, given by $\cup_{i =0}^r \mathcal{V}^i$, has rank $N$, for sufficiently large $r$.
\end{assumption} 

We will need another assumption on the regularity of $\Omega$. Towards this end an {\it admissible curve} $\gamma:[0,1] \rightarrow \Omega $ connecting two points $\x ,y \in \Omega$ is a Lipschitz curve in $\Omega$ for which there exist essentially bounded functions $u_i:[0,T]\rightarrow \mathbb{R}$ such that $\gamma$ is a solution of \eqref{eq:ctrsys} with $\gamma(0)=x$ and $\gamma(1) = y$.

\begin{definition}
\label{def:epsdel}
The domain $\Omega \subset \mathbb{R}^d$ is said to be  \textbf{non-characteristic} if for every $x \in \partial \Omega$, there exists a admissible curve $\gamma(t)$ such $\gamma(0) = x$ and $\gamma(t)$ is not tangential to $\partial \Omega$ at $x$.
\end{definition}
This definition imposes a regularity on the domain $\Omega$, which will be needed to apply the results of \cite{elamvazhuthi2023density} to conclude the invertibility of the operator $\sum_{i=1}^m Y_i^*Y_i$.

\begin{assumption}
\label{asmp1}
\textbf{(Compact Domain)}
In the compact case, we will make the following assumptions on the boundary of the domain $\Omega$.
\begin{enumerate}
	\item The domain $\Omega$ is compact and has a $C^2$ boundary $\partial \Omega$.
	\item The domain $\Omega$ is non-characteristic in the sense of Definition \ref{def:epsdel}.
\end{enumerate}
\end{assumption}

\begin{assumption}
\label{asmp3}
\textbf{(Smoothness of Density)}
The densities $p_{\rm target}, p_n \in C^{\infty}(\bar{{\Omega}})$ \rev{and $p_n$ is strictly positive on ${\Omega}$}.

\end{assumption}

We will say that $p \in L^2(\Omega)$ is a probability density, if $\int_{\Omega}p(x)dx =1 $ and $p$ is non-negative almost everywhere on $\Omega$.

Given these definitions we will show that we can find a feedback controller $u_i = \pi_i(t,x)$ such that the dynamical system tracks the forward process in reverse.

In order to state our main result, we will need to define some mathematical notions that will be used in this section. \rev{\rev{These notions provide the functional-analytic framework required to define the operators associated with the forward diffusion and to establish their invertibility and regularity properties, which are essential for the time-reversal and deterministic realizability analysis.}} 
We define $L^2(\Omega)$ as the space of square integrable functions over $\Omega$, where $\Omega \subset \mathbb{R}^d$ is an open, bounded and connected subset. The set of continuous functions $p$ for which $p_t$ lies in a Hilbert space $X$ will be referred to using $C([0,1];X)$. 

\textbf{Notation}
For a given real-valued function $a \in L^{\infty}(\Omega)$, $L^2_a(\Omega)$ refers to the set of all functions $f$ such that the norm of $f$ is defined as
$
\|f\|_a: = (\int_\Omega|f(x)|^2a(x)dx )^{{1/2}}< \infty.
$
We will always assume that the associated function $a$ is uniformly bounded from below by a positive constant, in which case the space $L^2_a(\Omega)$ is a Hilbert space with respect to the weighted inner product $\langle \cdot , \cdot \rangle_{a} :  L^2_a(\Omega) \times L^2_a(\Omega) \rightarrow \mathbb{R}$, given by
$
\langle f, g\rangle_{a} = \int_{\Omega} f(x){g}(x)a(x)dx
$
for each $f,g \in L^2_a(\Omega)$. When $a =\mathbf{1}$, where $\mathbf{1}$ is the function that takes the value $1$ almost everywhere on $\Omega$, the space $L^2_a(\Omega)$ coincides with the space $L^2(\Omega)$. Due to the assumptions on $a$, the spaces $L^2(\Omega)$ and $L^2_a(\Omega)$ are isomorphic and the same is true for the spaces $C([0,T];L^2(\Omega))$ and $C([0,T];L^2_a(\Omega)$. We will use this fact repeatedly. For a function $f \in L^2(\Omega)$ and a given constant $c$, we write $f \geq c$ to imply that $f$ is real-valued and $f(x)\geq c$ for almost every (a.e.) $x \in \Omega$. Let $L^{2}_{a,\perp}(\Omega) := \{ f \in L^2_a(\Omega);\int_{\Omega} f(x)dx = 0\}$ be the subspace of functions in $L^2_a(\Omega)$ that integrate to $0$. 
We define the {\it weighted horizontal Sobolev space} $WH^1_a(\Omega) = \big \lbrace f \in L^2(\Omega): Y_i(af) \in L^2(\Omega) \text{ for } 1 \leq i \leq m \big \rbrace$. We equip this space with the weighted horizontal Sobolev norm $\|\cdot\|_{WH^1_a}$, given by $\|f\|_{WH^1_a} = \Big( \|f\|^2_{2} + \sum_{i=1}^n\|  Y_i(af)\|^2_{2}\Big)^{1/2} \nonumber$ for each $f \in WH^1_a(\Omega)$. Here, the derivative action of $Y_i$ on a function $f$ is to be understood in the distributional sense. When $a=\mathbf{1}$, where $\mathbf{1}$ is the constant function that is equal to $1$ everywhere, we will denote $WH_a^1(\Omega)$ by $WH^1(\Omega)$.

Before we present our stability analysis, we give some more preliminary definitions.
Given $a\in L^{\infty}(\Omega)$ such that $a \geq c$ for some positive constant $c$, and $\mathcal{D}(\omega_a) = WH^1_a(\Omega)$, we define the sesquilinear form $\omega_a:\mathcal{D}(\omega_a) \times \mathcal{D}(\omega_a) \rightarrow \mathbb{R}$ as: $
\omega_a(u,v) = \sum\limits_{i= 1}^m\int\limits_{\Omega}  \frac{1}{a(x)} Y_i ( a (x) u(x)) \cdot Y_i ( a (x){v}(x))dx $.
for each $u \in \mathcal{D}(\omega_a)$. 
We associate with the form $\omega_a$ an operator $\mathcal{A}^a :\mathcal{D}(\mathcal{A}^a)  \rightarrow  L^2_a(\Omega)$, defined as $ \mathcal{A}^au = f$, if $\omega_a(u,v) = \langle f , v \rangle_a $ for all $v \in \mathcal{D}(\omega_a)$ and for all $u \in \mathcal{D}(\mathcal{A}^a) = \lbrace g \in \mathcal{D}(\omega_a): ~ \exists f \in L^2_a(\Omega) ~ \text{s.t.} ~  \omega_a(g,v)= \langle f, v \rangle_a ~ \forall v \in \mathcal{D}(\omega_a) \rbrace$. The operator $\mathcal{A}^a = -\sum_{i=1}Y_i^*(\frac{1}{a(x)}Y_i (a(x) \cdot))$. \rev{See Section \ref{subsec:comp} for the detailed computation establishing this fact.} We will use an alternative expression of the operator that will be useful in the computations later, and also connect to the operator classical Fokker-Planck operators.

We wish to express the operator $\mathcal{A}^a$ in a form similar to the well known Fokker-Plack operator $\Delta u = \nabla \cdot (V(x) u)$. Toward this end, Let $\mathcal{A}^a u =f$. If $a = \frac{1}{p_n}$ is the noise distribution, by Green's theorem we get,

The last expression must hold for all $v$. Therefore, formally, we can conclude that $u$ satisfies the boundary condition,
\begin{equation}
\sum_{i=1}^m\vec{n} \cdot (g_i Y_i ( \frac{u}{p_n(x)})) = 0
\end{equation}
Therefore, the PDE
\begin{eqnarray}
\label{eq:Mainsysan} 
\partial_t p =  - \sum_{i=1}^m  Y_i^*Y_i p   +\sum_{i=1}^m  Y_i^* ( p \frac{Y_ip_n}{p_n} )  ~~ in  ~~ \Omega \times [0,T], 
\end{eqnarray}
with zero flux boundary condition
\begin{equation}
\sum_{i=1}^m\vec{n} \cdot (g_i Y_i ( \frac{p}{p_n})) = 0
\label{eq:zbc}
\end{equation}
can expressed as an abstract ODE given by 
\begin{equation}
\dot{p} = \mathcal{A}^a p
\end{equation}

When $Y_i = \partial_{x_i}$ are coordinate vector-fields, the operator $\mathcal{A}^a$ is becomes the usual Fokker-Planck operator 
\[\Delta -\nabla \cdot (\nabla \log p_n ~~) =\Delta -\nabla \cdot (\frac{\nabla p_n}{p_n}),\] and for this special case we will denote the operator by $\B^a$.

We will need some different notions of solutions for the analysis further ahead. We will say that $p \in C([0,T]; L^2_a(\Omega))$ is a \textbf{mild solution} to \eqref{eq:Mainsysan} if 
\[ p_t = p_0 - \mathcal{A}^a\int_0^t p_sds\]
for all $ t \in (0,T]$.
Let $\pi:[0,T] \times \Rd \rightarrow \R^m$ be a control law. 

We will say that $p^c_t$ is a \textbf{classical solution} of \eqref{eq:fwdpdf_cs} if $p^c_t \in C^1((0,T] \times \Rd) $ if \eqref{eq:fwdpdf_cs} holds pointwise. Many times, we will need a weaker notion of solutions for this equation.  Toward this end, we will say that 
that $p^c_t$ is a \textbf{weak solution} to \eqref{eq:fwdpdf_cs} if
\begin{equation}
\label{eq:wkctty}
\int_0^1 \int_{\Rd} \big (\partial_t \phi(t,x) + g(x,\pi(t,x)) \cdot \nabla_x \phi(t,x) \big ) p^c_t(x)dxdt = 0
\end{equation}
for all $\phi \in C^{\infty}_c(\Rd \times (0,T))$.

Given these definitions, a number of properties of the {\it Nonholonomic} Fokker-Planck equation (NFE) \eqref{eq:Mainsysan} have been established in \cite{elamvazhuthi2023density}. For instance, it is known that the operator $\mathcal{A}^a$ generates a semigroup $T(t)$ of operators such that the solution satisfies $p_t = T(t) p_0$. Moreover, the asymptotic stability estimate holds 
\begin{equation}
\|\T(t)p_0-p_n\|_a \leq Me^{-\omega t} \|p_0 - p_n\|_a
\end{equation}
for all $t >0$. This is particularly important for the control approach proposed in the paper, as it gives a control law to transport the system towards the noise distribution.



The operator $\mathcal{A}^a $ is invertible. Using this, we can establish that any sufficiently regular trajectory on the set of probability densities can be tracked using the control system in reverse.

First, we state a general result that given any curve on the set of probability densities, for controllable driftless systems, one can find a control that exactly tracks the curve.

\begin{lemma}[\textbf{Exact tracking of positive densities}]
\label{lem:extrafull}
Given Assumption \ref{asmp2} and \ref{asmp1}, suppose $p^{\text{ref}} \in C([0,T];L^2(\Omega))$, $\dot{p}^{\text{ref}} \in C([0,T];L^2(\Omega))$, $ p^{\text{ref}}_t $ is a probability density and  $  p^{ref}_t > 0$,  for all $t\geq 0$. 
Then there exists a control law $\pi$ of the form
\[\pi_i = \frac{Y_i (\mathcal{A}^1)^{-1} \dot{p}^{ref}_t }{p^{ref}_{t}}\]
for all $t \in [0,T]$ and a weak solution $p^{\textup{c}}$ of \eqref{eq:fwdpdf_cs} satisfies 
\[p^{\textup{c}}_t = p^{\text{ref}}_t \text{ for all } t\in [0,T]\] 
\end{lemma}


Given the general result on tracking of probability densities, we can track a forward process in the reverse direction. This is the subject of the following Lemma, which addresses part of Problem \ref{prb3} for Algorithm 1.

\begin{lemma}[\textbf{Tracking the Holonomic Fokker-Planck Equation}]
\label{lem:extra}
Given Assumption \ref{asmp1}, 
suppose $p_0  \in L^2(\Omega)$. Let $p^{f}$ be the mild solution of \eqref{eq:Mainsysan} for $\mathcal{A}^a =\mathcal{B}^a$, then there exists a control law $\pi: [0,T] \times \Rd \rightarrow \R^m $ of the form
\[\pi_i(t,\cdot) = \frac{Y_i (\mathcal{A}^{1})^{-1} \dot{p}^f_{T-t} }{p^f_{T-t}}\]
such that a weak solution $p^{\textup{c}}$ of the \eqref{eq:fwdpdf_cs} satisfies 
\[p^{\textup{c}}_t = p^{\textup{f}}_{T-t} \text{ for all } t\in [0,T]\] 
\end{lemma}

The goal of the next result is to completely address Problem \eqref{prb3} for Algorithm 1. The idea is that in general, we do not know if the the vector field generated by the time-reversing control laws are Lipschitz. Hence, one cannot construct a probability measure by using the flow map associated with the vector-field. Nevertheless, the following result from optimal transport theory \cite{ambrosio2014continuity} enables the construction of such a measure. 

\begin{theorem}(\cite{ambrosio2014continuity})
\label{thm:sup}
\textbf{(Superposition solution)} 
Let $t \mapsto \mu_t \in \mathcal{P}(\Rd)$ be such that $\mu_t$ is a weak solution to the continuity equation according to \eqref{eq:wkctty} and 
\[ \int_0^T \int_{\Rd} \frac{|v(t,x)|}{1+|x|}d\mu_t(x) dt < \infty \]
Then there exist a probability measure $\mathbb{P} \in \mathcal{P}(\Rd \times C([0,T];\Rd))$ such that 
\begin{equation}
	\dot{\gamma}(t) = v(t,\gamma(t)),~~ \gamma(0) = y
\end{equation}
for $\mathbb{P}$ almost every $(y,\gamma) \in \Rd \times C([0,T];\Rd)$, and $\mathbb{P}(\gamma_t \in dx) = d\mu_t(x)$, for all $t \in [0,T]$.
\end{theorem}

Given this result we can construct a deterministic realization to the reverse process, as show in the following theorem. A key factor that helps us to verify the integrability requirement of the vector field is the regularity of the control law, which is due to the regularity of the solution of the inverses of Poisson equation as established in Proposition \ref{prop:invA}.

\begin{theorem}\textbf{(Determinsitic realization of reverse process)}
Given Assumption \ref{asmp1}. 
Suppose $p_0  \in L^2(\Omega)$ is a probability density. Let  $\pi_i $ be the control law given by 
\[\pi_i(t,\cdot) = \frac{Y_i \mathcal{A}^{-1} \dot{p}^c }{p^c_{t}}\]
for all $t \in (T,0]$
such that a weak solution $p^{\textup{c}}$ of the {\color{black} \eqref{eq:fwdpdf_cs}} satisfies 
\[p^{\textup{c}}_t = p^{\textup{f}}_{T-t} \text{ for all } t\in [0,T]\]
as shown by Lemma \ref{lem:extra}.
Then for every $\epsilon$ there exists a probability measure $\mathbb{P} \in \mathcal{P}(\Rd \times C([0,T-\epsilon];\Rd))$ such that 
\begin{equation}
	\dot \gamma_t = \sum_{i=1}^m  \pi_i(t,\gamma_t) g_i(\gamma_t),~~ \gamma(0) = y
\end{equation}
for $\mathbb{P}$ almost every $(y,\gamma) \in \Rd \times C([0,T];\Rd),$
for all $t \in [0,T]$ and, $\mathbb{P}(\gamma_t \in dx) = p^c_t(x)dx$, for all $t \in [0,T]$.

Moreover, if $p_0 \in \mathcal{D}(\mathcal{B}^1)$, then we can take $\epsilon =0$.
\label{thm:detrel}
\end{theorem}
Using this theorem on realization, we arrive at a corollary that addresses Problem \ref{prob:set_convergence}.

\rev{
\begin{corollary}
\label{cor:conafftargset1}
\textbf{(Target set control)}
Assume the hypothesis of Theorem \ref{thm:detrel}. If $p_0 \in \mathcal{D}(\mathcal{B}^1)$, and has a support contained in $\Omega_{\rm target} \subset \Omega$. Then $\mathbb{P}(\gamma_T \in \Omega_{\rm target}) = 1$. 
\end{corollary}
}

\subsection{Analysis of Algorithm 2 (Score Matching)}

In this section, we analyze the score matching approach used in Algorithm 2. In this case, we consider two different cases. When the system is driftless and the case of linear time invariant systems. Since, the case of the linear time invariant system comes with some restrictions in terms of the kind of noise distributions that one can converge to in the forward process, we consider these two cases separately.
\subsubsection{Driftless Systems}

In the first result, we derive what are the class of noise distributions that one can drive the forward process to asymptotcally and construct a closed form solution of the feedback control laws. 

\begin{theorem} (\textbf{Convergence of driftless system to noise distribution})
\label{thm:driftnoise}
Suppose the system \eqref{eq:ctrsys} is driftless. That is $X_0 \equiv 0$. Then for the choice of control law  
\begin{equation}
v_i = \sum_{j=1}^n\frac{\partial g_i^j}{\partial  x_j} +\frac{Y_ip_n}{p_n}
\end{equation}
the evolution of the probability density \eqref{eq:kfe} is equivalently described by the equation 
\begin{equation}
\label{eq:transforpde}
\partial_t p =   -\sum_{i=1}^m  Y_i^*Y_i p   +\sum_{i=1}^m  Y_i^* ( p \frac{Y_ip_n}{p_n} ) 
\end{equation}
Consequently, for the zero flux boundary condition \eqref{eq:zbc}, given $p_0 \in L^2(\Omega) \cap \mathcal{P
}(\Omega)$, there exist constants $M>0$ and $\lambda>0$ such that the solution of \eqref{eq:transforpde} satisfies
\begin{equation}
\|p_t - p_n\|_{2} \leq M e^{-\lambda t} \|p_0 - p_n\|_2
\end{equation}
for all $t \geq 0$.
\label{thm:ltmcon}
\end{theorem}

In the following theorem we establish that solutions of the nonholonomic Fokker-equation remain positive for all time $t>0$. The semigroup $\mathcal{T}^{\mathcal{A}^a}(t)$ generated by the operator $\mathcal{A}^a$ is {\it irreducible}. That is, if $p_0 \geq 0$, then $\mathcal{T}^{\mathcal{A}^a}(t)p_0>0$ for every $t>0$. Since the expression involves the solution $p^f_t$ in the denominator, this ensures that the time reversing control law computed further ahead is well posed.

\begin{proposition}
\label{prop:irred}
The semigroup $\mathcal{T}^{\mathcal{A}^a}(t)$ generated by the operator $\mathcal{A}^a$ is {\it irreducible}. That is, if $p_0 \geq 0$, then $\mathcal{T}^{\mathcal{A}^a}(t)p_0>0$ for every $t>0$.
\end{proposition}

Given this result of irreducibility, we can construct a deterministic control law that realizes the forward process. By reversibility of the continuity equation, this immediately also implies the control law can be used to reverse the forward process.

\begin{proposition}
\label{prop:realdrift}
Suppose the system is a controllable driftless system on a compact domain. Let $p^f_t$ be the mild solution of \eqref{eq:Mainsysan}. Consider the control law
\[\pi_i(t,\cdot) = \frac{Y_ip_n}{p_n}- \frac{Y_i p^f_{t}}{ p^f_{t}} \]
Then $p^f_t$ also solves the continuity equation in the weak sense as in definition \eqref{eq:wkctty}. 
\end{proposition}

Next, we consider Problem \ref{prb3}. Using the regularity of solutions of the Fokker-Plack equation established in Lemma \ref{lem:vecest1}, we can state the following result proving Problem \ref{prb3}.

\begin{theorem}
\label{thm:detcontaffalg2}
\textbf{Existence of deterministic reverse process}
Let $p_0 \in L^2(\Omega)$ be a probability density. Suppose the system is driftless and confined to a compact regular domain and satisfies Assumption \ref{asmp1}-\ref{asmp3}).
Then there exists a probability measure $\mathbb{P} \in \mathcal{P}(\Rd \times C([0,T];\Rd))$ such that 
\begin{equation}
\dot \gamma_t = \sum_{i=1}^m  \pi_i(t,\gamma_t) g_i(\gamma_t),~~ \gamma(0) = y
\end{equation}
for $\mathbb{P}$ almost every $(y,\gamma) \in \Rd \times C([0,T];\Rd)$,
where 
\[ \pi_i(t,x) =  -\frac{Y_ip_n}{p_n}+\frac{Y_i p^f_{T-t}}{ p^f_{T-t}} \]
for all $t \in [0,T]$.
\end{theorem}
\begin{proof}
Given the assumptions, the vector field 
\[v(t,x) =  \sum_{i=1}^m  \pi_i(t,x) g_i(x) \] 
satisfies the assumption of Theorem \ref{thm:sup} by results of Lemma \ref{lem:vecest1}.
\end{proof}

\rev{
\begin{corollary}
\textbf{(Target Set Control)}
Assume the hypothesis of Theorem \ref{thm:detcontaffalg2}. If $p_0 \in L^2(\Omega)$, and has a support contained in $\Omega_{\rm target} \subset \Omega$. Then $\mathbb{P}(\gamma_T \in \Omega_{\rm target}) = 1$. 
\label{cor:conafftargset2}
\end{corollary}
}

\subsubsection{Linear Time Invariant Systems}

Next, we consider the case of systems that are not driftless. While it is challenging for general systems with drift to design the equilibrium distribution, for stable linear time invariant systems, we can choose the noise distribution to be Gaussians. Toward this end we state the following assumption.

\begin{assumption}\textbf{(Controllable Stable Linear Time Invariant System (LTI))} The domain $\Omega = \mathbb{R}^d$ is the whole Euclidean space. There exist matrices $A \in \mathbb{R}^{d \times d}$ and $B \in \R^{d \times m}$ such that $g_0(x) = Ax$ and $[g_1(x),...,g_m(x) ] = Bx$ for all $x \in \Rd$. Additionally, ${\rm spec} ~A \subseteq \mathbb{C}^{-}$ where ${\rm spec} ~A$ denotes the spectrum of $A$ and $\mathbb{C}^{-}$ denotes the open left half of the complex plane. Additionally, the controllability Gramian,
\begin{equation}
Q(t) = \int_0^t e^{A\tau }BB^Te^{A^T \tau}d\tau
\end{equation}
is invertible for some (and hence all) $t>0$.
\label{asmp:lti}
\end{assumption}

Due to the presence of drift, several complexities are introduced in the problem. Firstly, one cannot in general make the forward process to converge to arbitrary distribution. For this reason, we consider the simpler case when $v_i(x)$ are set to $0$. In this case, for the solution of the PDE \eqref{eq:kfe}, rather than using semigroup theoretic arguments as in the driftless case, one can instead represent the solution using a integral operation of a kernel function given by,
\begin{equation}
\label{eq:kernel}
K(t,x,y) =  \frac{1}{(2 \pi)^{d/2} {\rm det}(Q_{t})^{1/2} }e^{-\frac{1}{2}(y-e^{At}x)^T Q^{-1}_{t} (y-e^{At}x)} 
\end{equation}
One can check by substitution that $K(t,x,y)$ is a classical solution of \eqref{eq:kfe} in $C^2((0,T];\Rd)$ for $p^f_0 = \delta_x$. For more general initial conditions $p^f_t$ is given by
\[p^f_t(y) = \int_{\Rd} K(t,x,y)p^f_0(x)dx \]
for all $t >0$. 
We will also need the kernel function 

Using this, we can see that
\[p^f_t(y) = \int_{\Rd} \tilde{K}(t,x,y) d(e^{At}_{\#}p^f_0)(x) \]
With an abuse of notation we denote
\[\tilde{K}(t,x) := \frac{1}{(2 \pi)^{d/2} {\rm det}(Q_{t})^{1/2} }e^{-\frac{1}{2}x^T Q^{-1}_{t} x}, \]
where $*$ denotes the convolution operation and hence,
\[p^f_t = \tilde{K}(t, \cdot)*\tilde{p}^f_t \]
where $\tilde{p}^f_t = e^{At}_{\#}p^f_0(x)$.

Unlike the case of driftless systems, it is not clear if the forward process can be stabilized to any given target noise distribution. However, the following result establishes a convergence for a specific gaussian distribution associated with the controllability Gramian of the system. While this result is well known, due to the lack of a specific reference, we state it here and prove it for convergence to the stationary distribution in $L_2$ norm, 

\begin{theorem}
(\textbf{Convergence of stable LTI system to stationary distribution})
\rev{Suppose the system satisfies Assumption \ref{asmp:lti}.}  Let $p_0 \in L^2(\Omega) \cap \mathcal{P}_2(\Omega)$ and $v_i \equiv 0$ in \eqref{eq:nhmsde}. Then the solution $p^f_t$ of
\eqref{eq:kfe} satisfies

\[ \lim_{t \rightarrow \infty} \|p^f_t - p_{\infty}\|_2 = 0\]
where 
\[p_{\infty}(x) = \frac{1}{(2 \pi)^{d/2} {\rm det}(Q_{\infty})^{1/2} }e^{-\frac{1}{2}x^T Q^{-1}_{\infty} x} \]
\label{thm:linstab}
\end{theorem}

We now consider Problem \ref{prb3} for Algorithm 2 and LTI systems. 

\begin{proposition}
\label{prop:linreadl}
Suppose the system satisfies Assumption \ref{asmp:lti}. Let $p^f_t$ be the mild classical solution of \eqref{eq:Mainsysan}. Consider the control law
\[\pi_i(t,x) = -\frac{Y_i p^f_{t}}{ p^f_{t}} \]
Then $p^f_t$ also solves the continuity equation in the weak sense as in definition \eqref{eq:wkctty}. 
\end{proposition}
\begin{proof}
This follows from the fact that $p^f_t$ is a classical solution of \eqref{eq:kfe} for $v_i \equiv 0$. Hence, by plugging in the expression for $\pi_i$ in the continuity equation, it is a solution of \eqref{eq:wkctty}.  
\end{proof}

Next, we use the established regularity result in conjunction with the superposition principle from optimal transport theory (Theorem \ref{thm:sup}) to establish a determistic realization of the reverse process as required in Problem \ref{problem:realization}.

\begin{theorem}
\label{thm:revglor2}
Let $p_0 \in L^2(\Omega)$ be a probability density with compact support. Suppose the system satisfies Assumption \ref{asmp:lti}. 
Then for every $\epsilon>0$, there exists a probability measure $\mathbb{P} \in \mathcal{P}(\Rd \times C([0,T-\epsilon];\Rd))$ such that 
\begin{equation}
\dot \gamma_t =A\gamma_t+ B \boldsymbol{\pi}(t,\gamma_t) ,~~ \gamma(0) = y
\end{equation}
for $\mathbb{P}$ almost every $(y,\gamma) \in \Rd \times C([0,T-\epsilon];\Rd)$,
where $\boldsymbol{\pi}(t,\gamma_t)$ is given by
\[ \pi_i(t,x) = \frac{Y_i p^f_{T-t}}{ p^f_{T-t}} \]
for all $t \in [0,T-\epsilon]$.
\end{theorem}
\rev{
Since in the linear case we are not able to set $\varepsilon = 0$ in Theorem~\ref{thm:revglor2}, we prove a slightly weaker claim addressing Problem~\ref{prob:set_convergence}. In particular, we cannot claim the existence of a probability measure on $C([0,T];\mathbb{R}^d)$ realizing the controlled dynamics up to time $T$, but only on $C([0,T);\mathbb{R}^d)$. Nevertheless, the marginals of $(\gamma_t)$, given by $p^f_{T-t}$ are well defined for all $t<T$ and converge weakly to $p_0$ as $t\rightarrow T$. 

On the one hand, this allows us to design a linear system for general control objectives such as multi-stability by choosing $p_0$ as a finite sum of Dirac measures. Complementarily, we can establish the following target set convergence result.


\begin{corollary}
\label{cor:trglin}
\textbf{(Target Set Control for LTI)}
Assume the hypotheses of Theorem~\ref{thm:revglor2}.
Let the support of $p_0$ be a contained in a compact set
$\Omega_{\rm target} \subset \Omega$.
Let $d_{\Omega_{\rm target}}:\mathbb{R}^d\to\mathbb{R}_{\ge 0}$ denote the distance function to
$\Omega_{\rm target}$.
Then, for every $\eta>0$,
\[
\lim_{t \rightarrow T}
\mathbb{P}\!\left(d_{\Omega_{\rm target}}(\gamma_t)>\eta\right)=0.
\]
\label{cor:lintargset}
\end{corollary}

} \section{Numerical Experiments}
\label{sec:numer}

We demonstrate the effectiveness of the proposed control algorithms on nonlinear dynamical systems\footnote{\href{https://github.com/darshangm/diffusion-nonlinear-control}{https://github.com/darshangm/diffusion-nonlinear-control}}. We further compare the two algorithms for different test beds. The numerical experiments were performed on a machine with Intel i9-9900K CPU with 128GB RAM and the Nvidia Quadro RTX 4000 GPU.

\subsubsection{Five-dimensional bilinear system}

We consider the following five-dimensional driftless system.
\begin{align}\label{eqn:fived-sys}
\begin{aligned}
\dot{x}_1 &= u_1, \ \ \dot{x}_2 = u_2, \ \ \dot{x}_3 = x_2 u_1, \\ \dot{x}_4 &= x_3 u_1, \ \ \dot{x}_5 = x_4 u_1
\end{aligned}
\end{align}
{\color{black} We sample $M$ data points from a Gaussian distribution $\mathcal{N}(0,0.5I) = p_{\text{target}}$. We set $V(x)=0$ for the forward process, therefore $p_n$ is the uniform distribution defined over the domain $\Omega= (-4,4)^5$ and consider the time horizon $[0,6]$.} The neural network used to estimate the controller $\texttt{NN}(t,x,\theta)$ for both algorithms has four hidden layers. We first demonstrate the effect of the number of training samples $M$ on the final estimated KL-divergence $\widehat{\text{KL}}(p^{\textup{c}}\big|p^{\textup{f}}_0)$. To test the final KL divergence, we apply the learned controller on 2000 uniformly sampled points in $\Omega$. We set the control task as one of stabilizing the Gaussian distribution $\mathcal{N}(\bf{0},0.5\bf{I})$.

In Figure~\ref{fig:five-dim-sys}(a) we show the KL divergence between the controlled density as given by the two algorithms and reference density to compare their performance. We see that diffusion-based algorithm 1 has a higher KL divergence in the middle of the horizon because it depends on careful sampling of the densities over the entire time horizon. The score-based model however performs better and achieves a denser distribution around the origin.

\begin{figure}[tbh]
\begin{tikzpicture}
\draw[blue, thick] (-2.8,2) -- (-1.8,2);
\node at (-0.8,2) {Algorithm 1};
\draw[red, thick] (1,2) -- (2,2);
\node at (3,2) {Algorithm 2};
\node (img1) at (-1.5,0) {\includegraphics[width=0.200\textwidth]{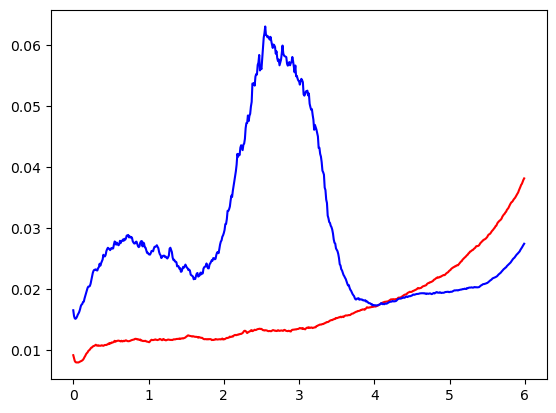}};
\node[above of= img1, node distance=0cm, yshift=-1.8cm,font=\color{black}]  {Training iterations};  
\node[above of= img1, node distance=0cm, yshift=-2.3cm,font=\color{black}]  {\small (a)};
\node[left of= img1, node distance=0cm, rotate=90, anchor=center,yshift=2.2cm,font=\color{black}] {KL divergence };
\node (img2) at  (2.8,0)  {\includegraphics[width=0.200\textwidth]{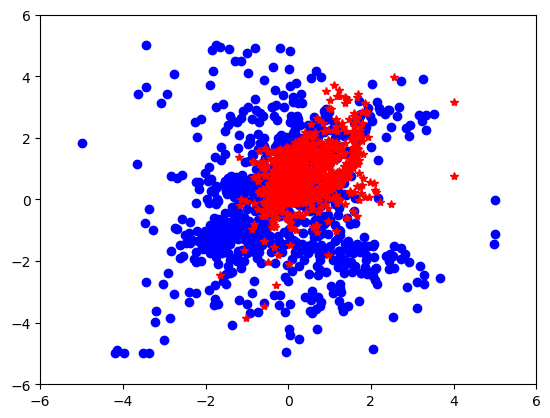}};
\node[above of= img2, node distance=0cm, yshift=-1.8cm,font=\color{black}]  {$x_1$}; 
\node[left of= img2, node distance=0cm, rotate=90, anchor=center,yshift=2.0cm,font=\color{black}] {$x_2$};
\node[above of= img2, node distance=0cm, yshift=-2.3cm,font=\color{black}]  {(b)};
\end{tikzpicture}
\caption{Experiments with a five-dimensional nonlinear system~\eqref{eqn:fived-sys}. (a) KL divergence between forward and reverse processes for algorithms 1 and 2. (b) Position of samples at final time step for algorithms 1 and 2.  }
\label{fig:five-dim-sys}
\end{figure}
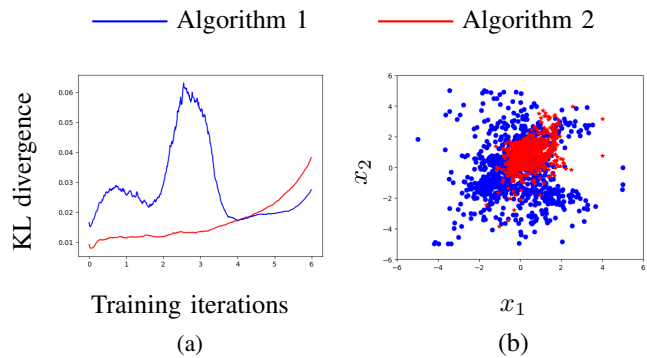

\subsubsection{Unicycle robot}
In this experiment, we consider the unicycle dynamics,
\begin{align}\label{eqn:unicycle}
\begin{aligned}
\dot{x}_1 &= u_1 \cos(x_3), \ \
\dot{x}_2 = u_1 \sin(x_3), \ \
\dot{x}_3 = u_2.    
\end{aligned}
\end{align}
We sample $M$ training samples from $p_\text{target} = \mathcal{N}(0,0.2I)$ and $p_n = \mathcal{N}(0,I)$. Note that this experiment demonstrates that our DDPM-based feedback algorithm can be applied for cases with $V(x) \neq 0$. Figure~\ref{fig:unicycle-expts}(a) depicts the performance of the controller for different number of measurement instances. It shows that with fewer measurement instances, it is more difficult for the controller to track the forward density. Figure~\ref{fig:unicycle-expts}(b) further validates our theoretical result even when $V(x)\neq0$. With higher number of training samples the neural network indeed learns to minimize the final KL divergence.

\begin{figure}[tbh]
\begin{tikzpicture}
\draw[blue, thick] (-2.8,2) -- (-1.8,2);
\node at (-0.8,2) {Algorithm 1};
\draw[red, thick] (1,2) -- (2,2);
\node at (3,2) {Algorithm 2};
\node (img1) at (-1.5,0) {\includegraphics[width=0.200\textwidth]{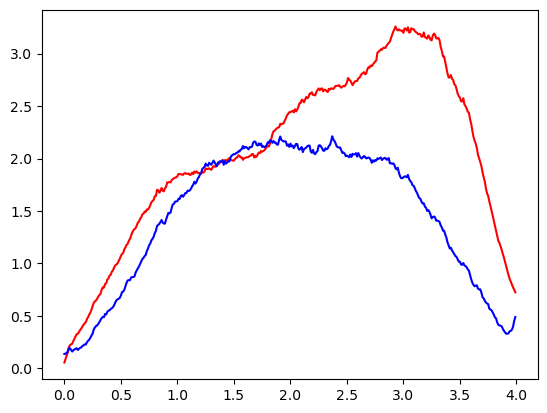}};
\node[above of= img1, node distance=0cm, yshift=-1.8cm,font=\color{black}]  {Training iterations};  
\node[above of= img1, node distance=0cm, yshift=-2.3cm,font=\color{black}]  {\small (a)};
\node[left of= img1, node distance=0cm, rotate=90, anchor=center,yshift=2.2cm,font=\color{black}] {KL divergence };
\node (img2) at  (2.8,0)  {\includegraphics[width=0.200\textwidth]{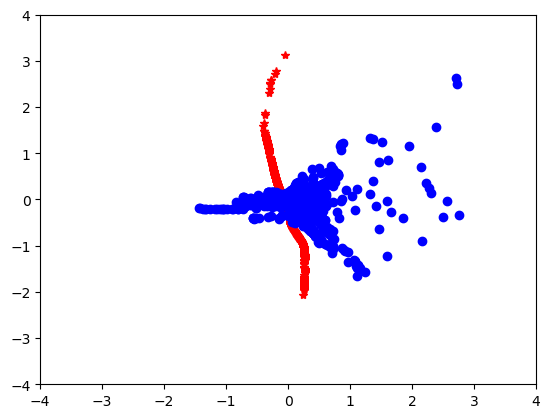}};
\node[above of= img2, node distance=0cm, yshift=-1.8cm,font=\color{black}]  {$x_1$}; 
\node[left of= img2, node distance=0cm, rotate=90, anchor=center,yshift=2.0cm,font=\color{black}] {$x_2$};
\node[above of= img2, node distance=0cm, yshift=-2.3cm,font=\color{black}]  {(b)};
\end{tikzpicture}
\caption{Experiments with unicycle dynamics~\eqref{eqn:unicycle} with {\color{black} $p_\text{target} = \mathcal{N}(0,0.2I)$.} (a) Final estimated KL divergence for different numbers of measurement instances $N$ vs. training iterations: shows that more measurement instances are required to achieve better feedback control when going from one Gaussian distribution to another, (b) Final KL divergence vs training iterations for different number of training samples: shows that the neural network can identify the controller with a sufficiently large number of training samples and state measuring instances.}
\label{fig:unicycle-expts}
\end{figure}
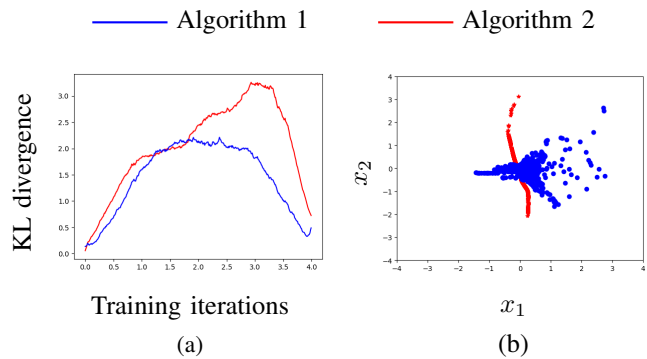

\subsubsection{Unicycle robot with obstacles}
\begin{figure*}[tbh]
\centering
\begin{tikzpicture}
\node (img1) at (-3.5,0) {\includegraphics[width=0.2800\textwidth]{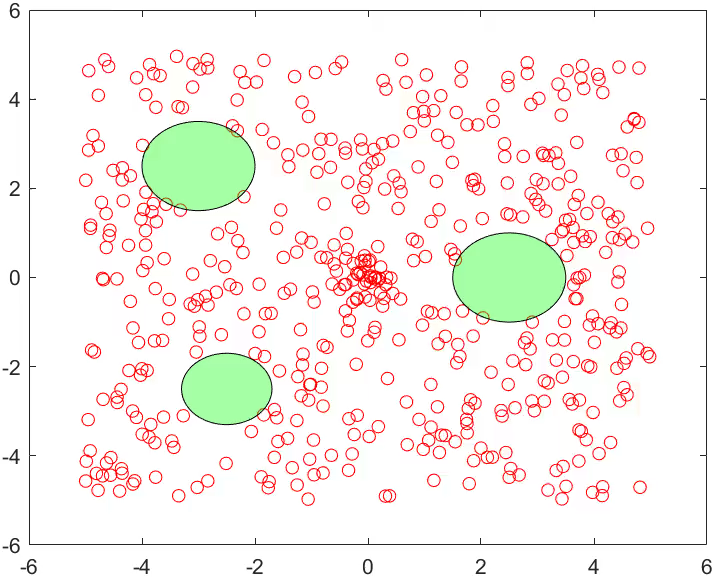}};
\node (img2) at  (2.5,0)  {\includegraphics[width=0.2800\textwidth]{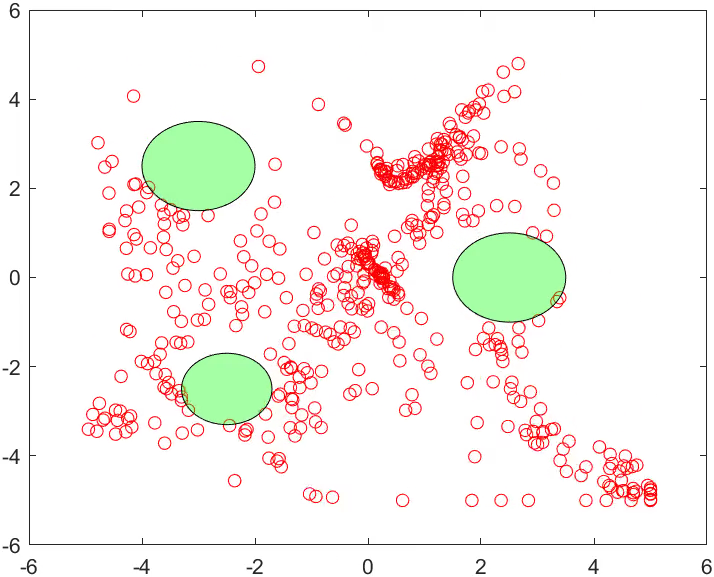}};
\node (img1) at (8.5,0) {\includegraphics[width=0.280\textwidth]{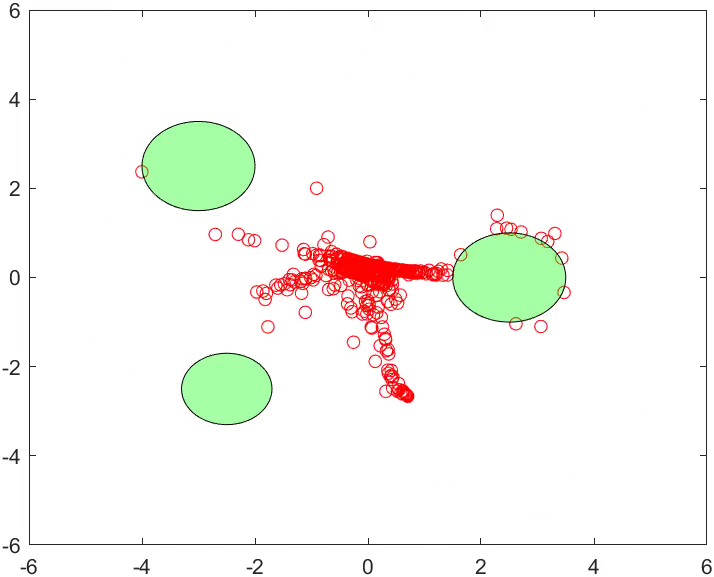}};  
\end{tikzpicture}
\caption{Experiments with unicycle dynamics ~\eqref{eqn:unicycle} with obstacles in the environment with {\color{black} $p_\text{target} = \mathcal{N}(0,0.2I)$.} The evolution of particles at different time instances in the control horizon of 10 seconds is depicted in the figure. It is evident that particles make use of the space between the obstacles to stabilize $p_\text{target}$.}
\label{fig:unicycle-obst}
\end{figure*}

In this experiment, we have the same unicycle dynamics for all the points. However, we include obstacles in the environment which are denoted by the green circles in Figure~\ref{fig:unicycle-obst}. We use algorithm 1 to train a controller for the unicycle robots to avoid obstacles and stabilize the Gaussian distribution centered at the origin. {\color{black} The noise distribution is taken to be the uniform distribution.} We consider a time horizon of 10 seconds. In this scenario, in the forward process, the unicycle robots are strictly not allowed to entire the regions around the obstacles using a reflection mechanism. That is, if by taking a control action, the point enters the obstacle, the point is moved back to the previous state. We show the evolution of the reverse process after the learned controller with the obstacles. At each time instance, we see the evolution of the particles making use of the space between the obstacles to stabilize the Gaussian distribution centered at the origin. \rev{The terminal distribution concentrates near the prescribed Gaussian parameters in terms of mean and covariance, but does not match the ideal normal density exactly. This might be due to inherent degeneracy of score matching loss, and we believe that alternative losses such as the flow matching \cite{elamvazhuthi2025flow} loss could potentially address this issue.}

\begin{figure}[htbp]
\begin{center}
\includegraphics[width=0.8\columnwidth]{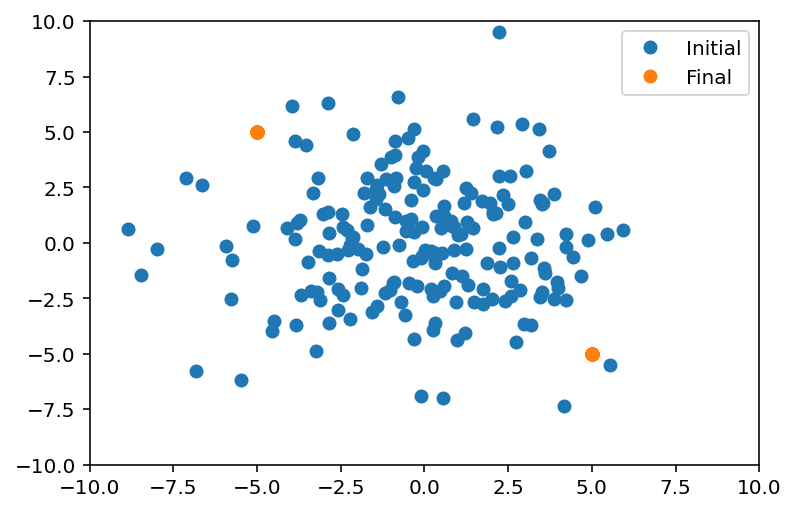}    
\end{center}
\caption{Experiments with four dimensional linear system \eqref{eqn:DI}} with $p_\text{target} =\frac{1}{2}\delta_{x^a} +  \frac{1}{2}\delta_{x^b} $ with $x^a= (-5,0,5,0)$ and  $x^b = (5,0,-5,0)$. The initial and final distributions of the particles are shown.
\label{fig:double-integrator}
\end{figure}
\subsubsection{Linear System} 
In this experiment,  we consider the linear double integrator system with unstable damping with two position cooridinates and two velocity coordinates
\begin{align}\label{eqn:DI}
\begin{aligned}
\dot{x}_1 &= x_2, \ \
\dot{x}_2 = x_1 + u_1,  \ \ \dot{x}_3 = x_2, \ \
\dot{x}_4 = x_3 + u_2 
\end{aligned}
\end{align}
The goal in this experiment is to stabilize the system to a sum of two Dirac measures $p_{\text{target}} =\frac{1}{2} \delta_{x^a} +  \frac{1}{2}\delta_{x^b} $at $x_a= (-5,0,5,0)$ and  $x^b = (5,0,-5,0)$. In this case, we use the controllability Gramian to compute the score function, and do not use any neural network. The distribution of the forward process $p^f_t$ is given in closed form as
\[p^f_t =  K*p_{\text{target}} = \frac{1}{2}K(t,x^a,y) + \frac{1}{2} K(t,x^b,y) \]
where the function $K$ is defined in \eqref{eq:kernel}. 
The initial conditions are sampled from \[p_{\infty}(x) =\frac{1}{(2 \pi)^{d/2} {\rm det}(Q_{\infty})^{1/2} }e^{-\frac{1}{2}x^T Q^{-1}_{\infty} x} \] 
where $Q_{\infty}$ is the infinite horizon controllability Gramian for the (open-loop stable) forward process with system $(A,B)$ matrices,

\begin{equation}
A = \begin{bmatrix}
0 & 1 & 0 & 0\\
0 & 1 & 0 & 0\\
0 &  0 & 0& 1\\
0 &  0 & 0 & 1\\
\end{bmatrix}, ~~ 
B = \begin{bmatrix}
0 & 0 \\
1 & 0\\
0 & 0\\
0 & 1\\
\end{bmatrix}.
\end{equation}
Note that for the forward system, with system matrices are $(-A,-B)$. In Figure \ref{fig:double-integrator}, we see that for this experiment the system is rendered bistable, and initial conditions are transported exactly to the target final states, depending on their initial conditions.

\rev{
\section{Conclusion}
We presented an approach to control nonlinear systems inspired by denoising diffusion-based generative modeling. We established the existence of deterministic realizations of reverse processes that eliminate the need for stochasticity in feedback synthesis. Potential future directions include proving that the reverse process is asymptotically stabilizing under more general conditions. From a numerical standpoint, it may be beneficial to augment the score matching loss to explicitly promote stability, as the standard formulation does not guarantee that the learned control laws inherit the stabilizing properties of the ideal reverse dynamics.
\appendix

\label{app:literature}

\subsection{Denoising Diffusion Probabilistic model}
	\label{sec:bgrnd}
    This section introduces Denoising Diffusion Probabilistic Models (DDPMs) \cite{song2020score}. DDPMs are a generative modeling technique that learns to sample from an unknown data density $p_\text{target}$ by (i) transforming it into a known  density $p_n$ from which one can easily sample (such density has been referred to as the \emph{noise density}), (ii) sampling from it, and finally (iii) reversing the transformation.  This is done using two stochastic differential equations. First, the {\it forward process} is a stochastic differential equation of the form
	\begin{align}
    \label{eq:sdeapp}
		\textup{d} X^\text{f} = V(X^\text{f})dt + \sqrt{2} \, \textup{d} W + n(X^f)\textup{d} Z,
	\end{align}
	where $X_0$ has density $p_\text{target}$. In \eqref{eq:sdeapp}, $W$ denotes the standard Brownian motion and $Z$ is a stochastic process that ensures that the state remains confined to a domain of interest $\Omega$. The vector field $\map{V}{\mathbb{R}^d}{\mathbb{R}^d}$ is chosen such that the density of the random variable $X$ converges to $p_n$. For example, if $V = \nabla \log p_n= \frac{\nabla p_n}{p_n}$ then $\lim_{t \rightarrow \infty} p_t = p_n$. Such convergence is guaranteed by the Fokker Planck equation that governs the evolution of the density $p$ of the state  $X$ of~\eqref{eq:sdeapp}: 
	\begin{align}
		\label{eq:fwdpdf}
		\partial_t p = \Delta p - \text{div}(V(x) p) = \text{div}  ([\nabla \log p -V(x)] p) . 
	\end{align}
	where  $\Delta : = \sum_{i=1}^d \partial_{x^2_i}$ and $\text{div} : = \sum_{i=1}^d \partial_{x_i}$ denote the Laplacian and the  divergence operator, and $p_0 = p_\text{target}$.
	
	
	When $\Omega$ is a strict subset of $\mathbb{R}^d$, this equation is additionally supplemented by a boundary condition, known as the {\it zero flux} boundary condition
	\begin{equation}
		\vec{n}(x)\cdot ( \nabla p _t(x) - V(x)) = 0 ~~~ \mbox{on} ~~~ \partial \Omega ,
	\end{equation}
	where $\vec{n}(x)$ is the unit vector normal to the boundary $\partial \Omega$ of the domain $\Omega$.
	This boundary condition ensures that $\int_{\mathbb{R}^d} p_t(x)dx = 1$ for all $t\geq 0$. An advantage of considering the situation of bounded domain is that one can choose $V\equiv 0$ and the noise density can be taken to be the uniform density on $\Omega$. This property will be useful when solving our control problem.
	

	The second part of the DDPM is the \emph{reverse process}, which aims
	to transport the density from $p_n$ back to $p_\text{target}$. There are multiple possible choices of the reverse process, including the {\it probabilistic flow ODE} given by
	\begin{align}\label{eq: reverse process}
		{\color{black} \textup{d}X^{\textup{r}} = \nabla \log p_{T-t}(X^{\textup{r}})\textup{d}t - V (X^{\textup{r}})\textup{d}t,} 
	\end{align}
	where $X^{\textup{r}}_0$ has density $p_n$ and $T$ is the horizon of the reverse process. 
	The evolution of the density $p^{\textup{r}}$ of the reverse process \eqref{eq: reverse process} is
	\begin{align}
		\partial_t p^{\textup{r}} = \text{div}  (-[\nabla \log p_{T-t} +V(x)] p^{\textup{r}}) .
	\end{align}
	In the ideal case, $p^r_T = p_{\rm target}$ and, $p_{T-t} = p^{\textup{r}}_t$ for all $t \in [0,T]$. 
	After simulating the forward process, one can learn the \emph{score} $\nabla \log p_{T-t}$ to run the reverse process to effectively sample from $p_{\text{target}}$.
	In practice, $p^{\textup{f}}_T \approx p_n$, and hence $p^{\textup{r}}_0  = p_n$, and one does not have complete information about the score. Usually, a neural network $\texttt{NN}(t,x,\theta)$ is used to approximate the score by solving the optimization problem
	\begin{align} 
		&\min_{\theta}  \int_0^T \mathbb{E}_{p_{T-t}} |\texttt{NN}(t,\cdot,\theta) -\nabla \log p_{T-t} |^2\textup{d}t \nonumber \\
		=&\min_{\theta}  \int_0^T \int_{\Rd} |\texttt{NN}(t,x,\theta) -\nabla \log p_{T-t}(x) |^2p_{T-t}(x)\textup{d}t 
	\end{align}
	This objective ensures that the solution $p^\theta_t$ of the equation 
	\begin{align}
		\partial_t p^\theta = \text{div}  ( [ \texttt{NN}(t,x,\theta) +V (x)] p^\theta),
	\end{align}
	where $p^\theta_0 = p_n$,  is close to $p_t$ so that we can sample from $p_{\text{target}}$ by running the reverse ODE,
	\begin{align} 
		\textup{d}X^\theta = \texttt{NN}(t,X^\theta,\theta)-  V (X^\theta)\textup{d}t 
	\end{align}
	such that $X_0$ is sampled from $p_{\text{noise}}$. 

\subsection{Computations} 
\label{subsec:comp}
In this section, we derive in detail the relation between the form $\omega_a(u,v)$ and the operator $\mathcal{A}^a $
\begin{align*}
&\omega_a(u,v) \\
& = \sum_{i=1}^m \int_{\Omega} p_n(x) Y_i \left( \frac{u(x)}{p_n(x)} \right) \cdot Y_i \left( \frac{v(x)}{p_n(x)} \right) \, dx \\
&= - \sum_{i=1}^m \int_{\Omega} Y_i^* \left( p_n(x) Y_i \left( \frac{u(x)}{p_n(x)} \right) \right) \cdot \frac{v(x)}{p_n(x)} \, dx \\
&\quad + \sum_{i=1}^m \int_{\partial \Omega} \vec{n} \cdot \left( g_i Y_i \left( \frac{v(x)}{p_n(x)} \right) \right) \frac{v(x)}{p_n(x)} \, dS \\
&= - \sum_{i=1}^m \int_{\Omega} Y_i^* \left( \frac{p_n(x) Y_i u(x) - u(x) Y_i p_n(x)}{p_n(x)} \right) \cdot \frac{v(x)}{p_n(x)} \, dx \\
&\quad + \sum_{i=1}^m \int_{\partial \Omega} \vec{n} \cdot \left( g_i Y_i \left( \frac{v(x)}{p_n(x)} \right) \right) \frac{v(x)}{p_n(x)} \, dS \\
&= - \sum_{i=1}^m \int_{\Omega} \left( Y_i^* Y_i u(x) - \frac{u(x) Y_i p_n(x)}{p_n(x)} \right) \cdot \frac{v(x)}{p_n(x)} \, dx \\
&\quad + \sum_{i=1}^m \int_{\partial \Omega} \vec{n} \cdot \left( g_i Y_i \left( \frac{u(x)}{p_n(x)} \right) \right) \frac{v(x)}{p_n(x)} \, dS \\
&= \langle \mathcal{A}^a u, v \rangle_a = \langle f, v \rangle_a
\end{align*}
for all $v \in WH^{a}(\Omega)$.
\subsection{Proofs}

\begin{proposition}[\textbf{Properties of the Nonholonomic Fokker-Planck equations}]
\label{prop:hteqprop}
Let $a = \frac{1}{p_n}$ for some function $p_n \in L^{\infty}(\Omega)$ such that $a \in L^{\infty}(\Omega)$. Given Assumption \ref{asmp1}, let $p_0 \in L^2_a(\Omega)$, then there exists a (mild) solution $p_t \in C([0,T];L^2_a(\Omega)) $ to the Fokker-Planck equation \eqref{eq:Mainsysan}. Moreover, the solution satisfies the following properties.
	There exists a semigroup of operators $(\T(t))_{t \geq 0}$ such that the solution of the \eqref{eq:fwdpdf} is given by $\T(t)p_0$.
    
	 If $p_{0} \in \D(\mathcal{A}^a)$, then $\dot{p} = -\mathcal{A}^a p$ for all $t \in [0,T]$, and $\dot{p} \in C((0,T];L^2_a(\Omega))$.
\end{proposition}

\begin{proof}[Proof of Proposition \ref{prop:hteqprop}]
The existence of the NFE has been shown in \cite[Theorem III.6]{elamvazhuthi2023density}. The semigroup $\mathcal{T}(t)$ is known to be analytic \cite[Theorem III.6]{elamvazhuthi2023density} ,and hence, the solution satisfies $p \in D(\mathcal{A}^a)$ for all $ t \in (0,T]$. Therefore, the differentiability of the solution with respect to time follows from \cite{curtain2012introduction}[Theorem 2.1.10].
\end{proof}

In the following proposition we establish the invertibility of the operator $\mathcal{A}^a$, which will play an important role in addressing Problem \ref{prb3}.

\begin{proposition} 
\label{prop:invA}
Given Assumption \ref{asmp2} and \ref{asmp1}, 
Suppose $f \in L^2_a(\Omega)$ such that $\int_{\Omega} f(x) dx =0$. Then there exists a unique solution $\psi \in WH^1_a(\Omega)$ to the weighted nonholonomic Poisson equation,
\begin{equation}
	\label{eq:hpoisseq}
	\mathcal{A}^a \psi = -\sum_{i=1}^m Y_i^*(\frac{1}{a(x)} Y_i(a(x)  \psi) = f.
\end{equation}
Hence, the restriction of the operator $\mathcal{A}^a$ on $L^{2}_{{a,\perp}}(\Omega)$, is invertible. Moreover, there exists a constant $C>0$ such that 
\begin{equation}
	\|Y_i(a\psi)\|_a \leq C \|f\|_a
\end{equation}
for all $f \in L^2_{{a,\perp}}(\Omega)$.
\end{proposition}
\begin{proof}
It is easy to see that $1/a$ is an eigenvector of $\mathcal{A}^a$ since $Y_i(a/a) = Y_i\mathbf{1}=\mathbf{0}$.  We know \cite{elamvazhuthi2023density} from that $1/a$ is an eigenvector with $0$ simple eigenvalue. Moreover, since the domain $\Omega$ satisfies Assumption \ref{asmp1} the domain is also NTA \cite{monti2005non}[Theorem 1.1], and hence also $\epsilon-\delta$ in the sense of \cite{garofalo1998lipschitz}. Therefore, from \cite{elamvazhuthi2023density}[Lemma III.3], the spectrum of $\mathcal{A}^a$ is purely discrete, consisting only of eigenvalues of finite multiplicity and have no finite accumulation point. Let $(\lambda_n)_{n=1}^{\infty}$ be the eigenvalues corresponding to the orthogonal basis of eigenvectors $\{e_n;n \in \mathbb{N}\}$. 

Since the operator $\mathcal{A}^a$ is positive and self-adjoint the eigenvalues are ordered in the form $0=\lambda_1< \lambda_2....$, with $\lim_{ n \rightarrow \infty} \lambda_n= +\infty$.

Let $u \in L^2_{a,\perp}(\Omega) $. Then $u = \sum_{n=2}^\infty \alpha_n e_n$ for some unique sequence $(\alpha_n)_{n=1}^{\infty}$ such that $\sum_{n=2}^{\infty} |\alpha_n|^2 = \|u\|^2_a$.
From this we can compute that
\begin{align}
	\label{eq:coerK}
	\langle \mathcal{A}^au,u \rangle_a = \sum_{n=2}^{\infty} \lambda_i |\alpha_n|^2 
	\geq   \frac{1}{\lambda_2}  \sum_{n=2}^{\infty}  |\alpha_n|^2 =  \frac{1}{\lambda_2}  \|u\|^2_2
\end{align}
Define the space $WH^1_{{a,\perp}}(\Omega) : = WH^1_a(\Omega) \cap L^2_{{a,\perp}}(\Omega)$ equiped with the norm $\|u\|_{\omega} : = \sqrt{\omega(u,u)}$.

The solvability of equation \eqref{eq:hpoisseq} is equivalent to finding a solution $\psi \in WH^1_{{a,\perp}}(\Omega) $ such that 
\[\omega_a(\psi,v) = <f,v>_a ,~~~\forall v \in \mathcal{D}(\omega)\]
It is easy to see that 
\[\omega(u,v) \leq \|u\|_\omega \|v\|_{\omega}\]
for all $u,v \in  WH^1_{{a,\perp}}(\Omega) $
Moreover, from the (Poincare inequality) estimate \eqref{eq:coerK} it follows that that 
\[\omega_a(u,u) \geq \frac{1}{\lambda_2}\|u\|_a\]
for all $u \in  WH^1_{{a,\perp}}(\Omega) $. Hence, existence and uniqueness of the solution $\psi \in WH^1_{{a,\perp}}(\Omega)$ follows from the Lax-Milgram theorem \cite{evans2022partial}.
\end{proof}

Proof of Theorem \ref{thm:detrel}
\begin{proof}
We compute,
\begin{align*}
	& \int_{0}^{T-\epsilon} \int_{\R^d} |\pi_i(t,x)| p^c_t (x)dxdt \\
	&= \int_{0}^{T-\epsilon} \int_{\R^d} | \frac{Y_i \mathcal{A}^{-1} \dot{p}^c }{p^c_{t}}| p^c_t (x)|dxdt \\
	& =\int_{0}^{T-\epsilon} \int_{\R^d} |Y_i \mathcal{A}^{-1} \dot{p}^c_t(x)|dxdt
\end{align*}
Since, $\dot{p}^f_t \in C(0,T]; L^2(\Omega))$ and hence $\dot{p}^c_t \in C[0,T); L^2(\Omega))$ by \ref{prop:invA}, we can conclude that this term is bounded for every $\epsilon >0$. Since the vector-fields $g_i$ are bounded on a compact set this implies that for $v(t,x) = \sum_{i=1}^m\pi_i(t,x) g_i(x)$ 
\[ \int_0^{T-\epsilon} \int_{\Rd} \frac{|v(t,x)|}{1+|x|}d\mu_t(x) dt < \infty .\]
Therefore, the existence of superposition solutions follows from \ref{thm:sup}.
If $p_0 \in \mathcal{D}(\mathcal{B}^1)$, then we can see that the integral estimate holds for $\epsilon = 0$, since in this case we have $\dot{p^c_t} \in C([0,T];L^2(\Omega))$. Therefore,  $\int_{\R^d} |Y_i \mathcal{A}^{-1} \dot{p}^c_t(x)|$ is uniformly bounded over $[0,T]$.
\end{proof}

{\it Proof of Lemma} \ref{lem:extrafull}. Let $\psi_t = (\mathcal{A}^1)^{-1} \dot{ p}^{ref}_t$, for all $ t \in [0,T]$. Then the control law law $\pi_i(t,x) = \frac{Y_i \psi_t(x)}{p^{ref}_t(x)}$ is well defined by Proposition \ref{prop:invA}.

Let $\phi \in C^{\infty}_c((0,T);\Rd)$. We compute
\begin{align*}
\int_0^T \int_{\Rd} \big (\partial_t \phi(t,x) + g(x,\pi(t,x)) \cdot \nabla_x \phi(t,x) \big ) p^{ref}_t(x)dxdt
\end{align*}
by substituting for $g$ this expression is equivalent to,
\begin{align*}
\int_0^T \int_{\Rd} \big (\partial_t \phi(t,x) - \sum_{i=1}^m \pi_t(t,x) \cdot Y_i \phi(t,x)  \big ) p^{ref}_t(x)dxdt
\end{align*}
Substituting for the form of the control law we get,
\begin{align*}
\int_0^T \int_{\Rd} \big (\partial_t \phi(t,x) - \sum_{i=1}  \frac{Y_i (\mathcal{A}^1)^{-1} \dot{p}^{ref} }{p^{ref}_{t}} \cdot Y_i \phi(t,x) \big ) p^{ref}_t(x)dxdt 
\end{align*}
Since $\sum_{i=1}^m Y^*_iY_i = \mathcal{A}^1$, this becomes
\begin{align*}
\int_0^T \int_{\Rd} \big (\partial_t \phi(t,x)p^{ref}_t(x) + \dot{p}^{ref}_t   \phi(t,x)  ) dxdt 
\end{align*}
By integration by parts with respect to time, and using the fact that $\phi$ is compactly supported, this term is equal to $0$. This concludes that $p^{ref}_t$ satisfies the continuity equation.

\hfill $\qed$

{\it Proof of Lemma} \ref{lem:extra}
It is a well known property of the solutions of the
Fokker-Planck equation that $p^f_t > 0$ for all $t \in (0,T]$. 
Moreover, $p^f_t \in \D(\B^a)$ and we know that $ \dot{p}^f_t = -\B^a p^f_t $ for all $(0,T]$ from Proposition \ref{prop:hteqprop}. This implies that $\langle \mathbf{1}, \dot{p}^f_t \rangle_2 = \langle \B^a \mathbf{1}, \dot{p}^f_t\rangle_2 = \mathbf{0}$.
Hence, $\dot{p}^f_t \in L^2_{{a,\perp}} (\Omega)$ for all $ t \in [0,T]$. Additionally, we know from Proposition \ref{prop:hteqprop} that $\dot{p}^f_t \in C((0,T];L^2_a(\Omega))$.

\hfill $\qed$

\begin{lemma}
\label{lem:vecest1}
Given $p_0 \in L^2(\Omega) \cap \mathcal{P}({\Omega})$. Let $p^f_t$ be the mild solution of \eqref{eq:kolmog2}. Then we have that 
\begin{equation}
\int_0^T \int_{\Omega}|Y_i p^f_t|dxdt < \infty
\end{equation}
\end{lemma}
\begin{proof}
We know that $p^f_t \in \mathcal{D}(\mathcal{A}^a)$ for all $t >0$. Since $\mathcal{D}(\omega^a) \subseteq \mathcal{D}(\mathcal{A}^a)$ we can infer that $p^f_t \in WH^1_a(\Omega)$ for all $t >0$. Moreover, 
\begin{equation}
\frac{d}{dt} \|p^f_t\|^2_a = -\omega_a(p^f_t,p^f_t) 
\end{equation}
for all $t >0$. Fix $\epsilon >0$. This implies that 
\begin{equation}
\label{eq:esti}
\|p^f_T\|_2 -\|p^f_\epsilon\|_2  = -\int^T_\epsilon \omega_a(p^f_t,p^f_t)dt 
\end{equation}
It follows from a simple application of product rule that $f \in WH^1_a(\Omega)$ implies $f \in WH^1_1(\Omega)$ since $ a \in C^{1}(\bar{\Omega})$. Moreover, if $f \in WH^1_{a}(\Omega)$, then $f \in WH^1_1(\Omega)$ with 
\[ \|f\|_{WH^1_1(\Omega)} \leq C \|f\|_{WH^1_a(\Omega)}\]
for some constant $C>0$ independent of $f$. 
Since, $p^f \in C([0,T];L^2(\Omega))$ we can combine this with \eqref{eq:esti} that 
\[\int_\epsilon^T \int_{\Omega} |Y_ip^f_t|^2dxdt < C \sup_{t \in [0,T]} \|p^f_t\|_{2} \]
for some constant $C>0$. Since the set $\Omega$ is compact this implies that 
\[\int_\epsilon^T \int_{\Omega} |Y_ip^f_t|dxdt < C' \sup_{t \in [0,T]} \|p^f_t\|_{2} \]
for some constant $C' > 0$. Since the constant is independent of $\epsilon>0$, the result follows. 
\end{proof}

Proof of Theorem \ref{thm:driftnoise}

\begin{proof}
We require that $v_i(x)$ is such that 
\begin{align}
&\sum_{i=1}^m (Y_i^*)^2 p + \sum_{i=1}^m Y_i^*(v_i(x) p) \\
& = \sum_{i=1}^m - (Y_i^* Y_i) p + \sum_{i=1}^m Y_i^* \left( p \frac{Y_i p_n}{p_n} \right).
\end{align}

Firstly, by the product rule:
\begin{align*}
Y_i^* &= \sum_{j=1}^n -\frac{\partial}{\partial x_j} \left( g_i^j \cdot \right) \\
&= \sum_{j=1}^n -\frac{\partial g_i^j}{\partial x_j} \left( \cdot \right) - \sum_{j=1}^n g_i^j \frac{\partial}{\partial x_j} \left( \cdot \right) \\
&= \sum_{j=1}^n -\frac{\partial g_i^j}{\partial x_j} \left( \cdot \right) - Y_i.
\end{align*}

Therefore, we have:
\begin{align*}
Y_i = \sum_{j=1}^n -\frac{\partial g_i^j}{\partial x_j} \left( \cdot \right) - Y_i^*.
\end{align*}

If we set 
\begin{equation*}
v_i = \sum_{j=1}^n \frac{\partial g_i^j}{\partial x_j} + \frac{Y_i p_n}{p_n},
\end{equation*}
we get:
\begin{align}
&\sum_{i=1}^m (Y_i^*)^2 + \sum_{i=1}^m Y_i^* \left( -Y_i - Y_i^* \right) + \sum_{i=1}^m Y_i^* \left( \cdot \frac{Y_i p_n}{p_n} \right) \nonumber \\
&= - \sum_{i=1}^m Y_i^* Y_i + \sum_{i=1}^m Y_i^* \left( \cdot \frac{Y_i p_n}{p_n} \right).
\end{align}

Then, equation \eqref{eq:kfe} becomes:
\begin{equation}
\partial_t p = - \sum_{i=1}^m Y_i^* Y_i p + \sum_{i=1}^m Y_i^* \left( p \frac{Y_i p_n}{p_n} \right). \label{eq:nonhomhe}
\end{equation}
The (mild) solution to this equation is given by the semigroup generated by the operator $-\mathcal{A}^a$ for $a = 1/p_n$. Hence, the result follows from Proposition \ref{prop:hteqprop}.
\end{proof}

Proof of Proposition \ref{prop:irred}
\begin{proof}
For any any measurable set $\tilde{\Omega} \subseteq \Omega$, let $\xi_{\tilde{\Omega}}$ denote the characteristic function. According to \cite[Corollary 2.11]{ouhabaz2009analysis} irreducibility of the semigroup is equivalent to the fact that there exists $\Omega_1 \subset \Omega$ such that the Lebesgue measure of $\Omega \backslash \Omega_1$ is not equal to $0$ and whenever $u$ lies in form domain $\mathcal{D}(\omega_a)$ we also have that $\xi_{\Omega_1} u  $ lies in the domain of the form $\mathcal{D}(\omega_a)$. Clearly, $1/a$ lies in the form domain. Suppose the semigroup is not irreducible. Then $v = \xi_{\Omega_1} 1/a$ must lie in the form domain for some set $\Omega_1$. We know that $v \in \mathcal{D}(\omega_a)$ implies $\xi_{\Omega_1} \in \mathcal{D}(\omega_\mathbf{1})$. It is a property of weak derivatives that $\xi_{\{v=0\}}Y_iv =0$, where $\{v=0\}=:\{x \in \Omega; v(x) = 0 \}$. This implies that $Y_i \xi_{\Omega_1} = \xi_{\Omega_1}Y_i\xi_{\Omega} = \xi_{\Omega_1}Y_i \mathbf{1} = 0$. Therefore, $\omega_1(\xi_{\Omega_1},\xi_{\Omega_1}) = 0$ and hence $\mathcal{A}^{\mathbf{1}} \xi_{\Omega_1} = \mathbf{0}$. This contradicts with \cite[Theorem 3.6]{elamvazhuthi2023density} that $0$ is a simple eigenvalue of $\mathcal{A}^{\mathbf{1}}$, unless $\Omega_1$ has either full Lebesgue measure or 0 Lebesgue measure. 
\end{proof}

Proof of Proposition \ref{prop:realdrift}

\begin{proof}
The expression is well posed by Proposition \ref{prop:irred}.
Let $\phi \in C^{\infty}_c((0,T) \times \Rd)$. By substituting for $\pi_i$ in the definition \eqref{eq:wkctty}, we compute,  \ref{thm:driftnoise},
\begin{align}
& \int_{\Rd} \big (\partial_t \phi(t,x) + (g(x,\pi(t,x)))  \cdot \phi(t,x) \big )p^f_t(x)dx \\
& = \int_{\Rd} \big (\partial_t \phi(t,x) + \sum_{i=1}^m(\pi_i(t,x) \cdot Y_i \phi(t,x) \big )p^f_t(x)dx  \nonumber \\
& =\int_{\Rd} \partial_t \phi(t,x) + \sum_{i}(\frac{Y_i p^f_t}{p^f_t} - \frac{Y_i p_n}{p_n} ) \cdot Y_i \phi(t,x) \big )p^f_t(x)dx  \nonumber \\
& =  \int_{\Rd} \partial_t \phi(t,x) p^f_t(x)dx + 
\int_{\Rd}\sum_{i}(\frac{Y_i p^f_t}{p^f_t} - \frac{Y_i p_n}{p_n} ) \cdot Y_i \phi(t,x) \big )dx  \nonumber \\
& = \int_{\Rd} \big (\partial_t \phi(t,x) p^f_t(x)dx + \int_0^T\omega_a(p^f_t,\phi(\cdot,t))
\label{eq:2ndtotst}
\end{align}
for all $ t\in (0,T)$
By integration by parts as using the fact that $\phi$ is compactly supported,
\begin{align*}
\int_0^T\int_{\Rd} \partial_t \phi(t,x) p^f_t(x)dx  & =  -\int_0^T\int_{\Rd} \phi(t,x) \frac{d}{dt}p^f_t(x)dx \\
& = -\int_0^T \omega_a(p^f_t,\phi(t,\cdot))
\end{align*}
Substituting this expression in \eqref{eq:2ndtotst} we get that,
\[\int_0^T \int_{\Rd} \big (\partial_t \phi(t,x) + (g(x,\pi(t,x)))  \cdot \phi(t,x) \big )p^f_t(x)dxdt =0.\]
This concludes the proof.
\end{proof}

Proof of Theorem \ref{thm:linstab}

\begin{proof}

Now we can compute the error,
\begin{align}
\|p_t - p_n\|^2_2 &= \| \int_{\Rd} K(t,x,\cdot)p^f_0(x)dx-p_n\|^2_2 \\
& = \int_{\Rd} |\int_{\Rd} K(t,x,y)p^f_0(x)dx-p_n(y)|^2dy 
\label{eq:doubdct}
\end{align}
Since, the system is a stable LTI, $K$ is uniformly bounded with time. Moreover, since $A$ is Hurwitz and $e^{At }x \rightarrow 0$ for all $x \in \Rd$, we know that
\[ \lim_{t \rightarrow \infty} K(t,x,y)p^f_0(x) = \lim_{t \rightarrow \infty} p_{n}(y) p^f_0(x) \] 
for all $x \in \Rd$.
From the dominated convergence theorem, this implies that 
\[\lim_{t \rightarrow \infty}|\int_{\Rd} K(t,x,y)p^f_0(x)dx-p_n(y)|^2dy = \int_{\Rd} p_n(y)p^f_0(x)dx  \]
Using another application of dominated convergence theorem to the outer intergal of \eqref{eq:doubdct}, we can conclude that 
\begin{align*}
\|p_t - p_n\|^2_2 &= \int_{\Rd} \lim_
{t \rightarrow \infty}|\int_{\Rd} K(t,x,y)p^f_0(x)dx-p_n(y)|^2dy \\
& = \int_{\Rd} |\int_{\Rd} p_{\infty}(y) p^f_0(x)dx - p_{\infty}(y)|^2dy
\\
& = 0
\end{align*}
\end{proof}

Next, we derive some derivative bounds on the solutions that will aide addressing Problem \ref{problem:realization}.

\begin{theorem}
\label{thm:score_linear_growth}
Suppose the system satisfies  Assumption~\ref{asmp:lti}, and let $v_i \equiv 0$ in~\eqref{eq:nhmsde}. Let $p_t^f$ be the solution of~\eqref{eq:kfe}. Fix $\epsilon>0$ and assume that the initial density $p_0^f$ has compact support in some ball 
 $\overline{B_R(0)}$ for some $R>0$.
Then, for every $t\in[\epsilon,T]$ and all $y\in\mathbb R^d$, we have  the pointwise linear-growth bound
\[
\bigl|\nabla \log p_t^f(y)\bigr|
\;\le\;
\|Q_t^{-1}\|\bigl(|y|+\|e^{At}\|\,R\bigr).
\]
\end{theorem}

\begin{proof}
As noted previously, the solution of the PDE \eqref{eq:kfe} can be represented using a integral operation of a kernel function given by
\[
p_t^f(y)=\int_{\mathbb R^d}K(t,x,y)\,p_0^f(x)\,dx,
\]
where
\[K(t,x,y) =  \frac{1}{(2 \pi)^{d/2} {\rm det}(Q_{t})^{1/2} }e^{-\frac{1}{2}(y-e^{At}x)^T Q^{-1}_{t} (y-e^{At}x)}  \]

Differentiating with respect to $y$ yields
\[
\nabla_y K(t,x,y)=-K(t,x,y)\,Q_t^{-1}(y-e^{At}x),
\]
and hence
\begin{align*}
\nabla p_t^f(y)
&=\int_{\mathbb R^d}\nabla_y K(t,x,y)\,p_0^f(x)\,dx \\
&=-Q_t^{-1}\int_{\mathbb R^d}K(t,x,y)\,(y-e^{At}x)\,p_0^f(x)\,dx.
\end{align*}
Since $Q_t > 0$ for all $t>0$ by controllability, $p_t^f(y)>0$ for all $y$, and therefore
\begin{align*}
\nabla \log p_t^f(y)
& =\frac{\nabla p_t^f(y)}{p_t^f(y)} \\
& =-Q_t^{-1}\left(
y-\frac{\int_{\mathbb R^d}K(t,x,y)\,e^{At}x\,p_0^f(x)\,dx}
{\int_{\mathbb R^d}K(t,x,y)\,p_0^f(x)\,dx}
\right).
\end{align*}
The matrix $e^{At}$ has finite matrix norm and $p^f_0$ has compact support in $R$. Therefore, $\|e^{At} x\|$ can be bounded as $R\|e^{At} \|$. Then using a simple application of triangle inequality and the fact that $\int_{\mathbb R^d}K(t,x,y)\,p_0^f(x)\,dx = 1$, the result follows.
\end{proof}

Proof of Theorem \ref{thm:revglor2}

\begin{proof}
Consider the vector field
\[v(t,x) = A\gamma_t+ B \boldsymbol{\pi}(t,\gamma_t)  \] 
We compute
\begin{align*}
& \int_0^T \int_{\Rd} \frac{|v(t,x)|}{1+|x|}p^f_{T-t}(x)dx dt  \\
&=  \int_0^T \int_{\Rd}  \frac{|Ax|}{1+|x|} p^f_{T-t}(x)dx dt \\
&+ \int_0^T \int_{\Rd}  \frac{|B \boldsymbol{\pi}(t,\gamma_t)|}{1+|x|}  p^f_{T-t}(x)dx dt\\
&\leq   C \int_0^T \int_{\Rd}  \frac{|x|}{1+|x|} p^f_{T-t}(x)dx dt \\
&+ C \sum_{i=1}^m\int_0^T \int_{\Rd}  \frac{| Y_ip^f_{T-t}(t,\gamma_t)|}{1+|x|} dx dt
\end{align*}
for some constant $C>0$ depending only on the matrices $A$ and $B$.
The first term is clearly bounded. The second term is bounded by the derivative bound of Theorem \ref{thm:score_linear_growth}. Then the result follows from the superposition principle stated in Theorem \ref{thm:sup}. Note that for this case since we have linear growth of vector fields, one in fact, doesn't need to invoke the superposition principle. However, to keep the discussion unified we maintain the presentation this way. 
\end{proof}

Proof of Corollary \ref{cor:trglin}.

\begin{proof}
Due to the linear growth condition established in Theorem~\ref{thm:score_linear_growth},
solutions of the ODE for the time-reversing control are globally defined over the interval $[0,T)$.
This implies that the probability measure $\mathbb{P}$ in Theorem~\ref{thm:revglor2} is well defined
on $C([0,T);\mathbb{R}^d)$ and unique. Moreover, by construction we have
\begin{align*}
(\gamma_t)_\#\mathbb{P} = p^f_{T-t}, \qquad \forall t\in[0,T),
\end{align*}
and hence $(\gamma_t)_\#\mathbb{P}$ converges weakly to $p_0$ as $t\rightarrow T$.
Since $\Omega_{\rm target}$ is closed, $d_{\Omega_{\rm target}}$ is continuous. Therefore,
\begin{align*}
d_{\Omega_{\rm target}}(\gamma_t)
~\rightarrow~
(d_{\Omega_{\rm target}})_\# p_0
\qquad \text{as } t\rightarrow T.
\end{align*}
Because $\mathrm{supp}(p_0)\subset\Omega_{\rm target}$, we have
\[
d_{\Omega_{\rm target}}(x)=0
\quad \text{for } p_0\text{-a.e. } x,
\]
and hence
$
(d_{\Omega_{\rm target}})_\# p_0 = \delta_0.
$
Consequently, $
d_{\Omega_{\rm target}}(\gamma_t) ~\Rightarrow~ 0.$
Since convergence in distribution to a constant implies convergence in probability \cite{billingsley2013convergence}[Theorem 25.3], it follows that
\begin{align*}
\lim_{t \rightarrow T}
\mathbb{P}\!\left(d_{\Omega_{\rm target}}(\gamma_t)>\eta\right)=0,
\qquad \forall \eta>0,
\end{align*}
which completes the proof.
\end{proof}
}



%
\bibliographystyle{unsrt}
\bibliography{ref}

@book{bakry2013analysis,
  title={Analysis and geometry of Markov diffusion operators},
  author={Bakry, Dominique and Gentil, Ivan and Ledoux, Michel},
  volume={348},
  year={2013},
  publisher={Springer Science \& Business Media}
}

@book{oksendal2013stochastic,
  title={Stochastic differential equations: an introduction with applications},
  author={Oksendal, Bernt},
  year={2013},
  publisher={Springer Science \& Business Media}
}

@article{elamvazhuthi2025flow,
  title={Flow Matching for Measure Transport and Feedback Stabilization of Control-Affine Systems},
  author={Elamvazhuthi, Karthik},
  journal={arXiv preprint arXiv:2510.02706},
  year={2025}
}

@article{grong2024score,
  title={Score matching for sub-Riemannian bridge sampling},
  author={Grong, Erlend and Habermann, Karen and Sommer, Stefan},
  journal={arXiv preprint arXiv:2404.15258},
  year={2024}
}

@article{mei2025time,
  title={A Time-Reversal Control Synthesis for Steering the State of Stochastic Systems},
  author={Mei, Yuhang and Taghvaei, Amirhossein and Pakniyat, Ali},
  journal={arXiv preprint arXiv:2504.00238},
  year={2025}
}

@article{brockett1983asymptotic,
  title={Asymptotic stability and feedback stabilization},
  author={Brockett, Roger W and others},
  journal={Differential geometric control theory},
  volume={27},
  number={1},
  pages={181--191},
  year={1983},
  publisher={Boston}
}

@book{clarke2013functional,
  title={Functional analysis, calculus of variations and optimal control},
  author={Clarke, Francis},
  volume={264},
  year={2013},
  publisher={Springer}
}

@book{khalil2002nonlinear,
  title={Nonlinear systems},
  author={Khalil, Hassan K and Grizzle, Jessy W},
  volume={3},
  year={2002},
  publisher={Prentice hall Upper Saddle River, NJ}
}

@article{ambrosio2014continuity,
  title={Continuity equations and ODE flows with non-smooth velocity},
  author={Ambrosio, Luigi and Crippa, Gianluca},
  journal={Proceedings of the Royal Society of Edinburgh Section A: Mathematics},
  volume={144},
  number={6},
  pages={1191--1244},
  year={2014},
  publisher={Royal Society of Edinburgh Scotland Foundation}
}

@article{cattiaux1988time,
  title={Time reversal of diffusion processes with a boundary condition},
  author={Cattiaux, Patrick},
  journal={Stochastic processes and their Applications},
  volume={28},
  number={2},
  pages={275--292},
  year={1988},
  publisher={Elsevier}
}

@article{haussmann1986time,
  title={Time reversal of diffusions},
  author={Haussmann, Ulrich G and Pardoux, Etienne},
  journal={The Annals of Probability},
  pages={1188--1205},
  year={1986},
  publisher={JSTOR}
}

@article{santambrogio2015optimal,
  title={Optimal transport for applied mathematicians},
  author={Santambrogio, Filippo},
  journal={Birk{\"a}user, NY},
  volume={55},
  number={58-63},
  pages={94},
  year={2015},
  publisher={Springer}
}

@book{ouhabaz2009analysis,
  title={Analysis of heat equations on domains.(LMS-31)},
  author={Ouhabaz, El-Maati},
  year={2009},
  publisher={Princeton University Press}
}

@article{monti2005non,
  title={Non-tangentially accessible domains for vector fields},
  author={Monti, Roberto and Morbidelli, Daniele},
  journal={Indiana University mathematics journal},
  pages={473--498},
  year={2005},
  publisher={JSTOR}
}

@article{garofalo1998lipschitz,
  title={Lipschitz continuity, global smooth approximations and extension theorems for Sobolev functions in Carnot-Carath{\'e}odory spaces},
  author={Garofalo, Nicola and Nhieu, Duy-Minh},
  journal={Journal d’Analyse Math{\'e}matique},
  volume={74},
  number={1},
  pages={67--97},
  year={1998},
  publisher={Springer-Verlag New York}
}

@book{curtain2012introduction,
  title={An introduction to infinite-dimensional linear systems theory},
  author={Curtain, Ruth F and Zwart, Hans},
  volume={21},
  year={2012},
  publisher={Springer Science \& Business Media}
}

@article{elamvazhuthi2023dynamical,
  title={Dynamical optimal transport of nonlinear control-affine systems},
  author={Elamvazhuthi, Karthik and Liu, Siting and Li, Wuchen and Osher, Stanley},
  journal={Journal of Computational Dynamics},
  volume={10},
  number={4},
  pages={425--449},
  year={2023},
  publisher={Journal of Computational Dynamics}
}

@article{agrachev2009optimal,
  title={Optimal transportation under nonholonomic constraints},
  author={Agrachev, Andrei and Lee, Paul},
  journal={Transactions of the American Mathematical Society},
  volume={361},
  number={11},
  pages={6019--6047},
  year={2009}
}

@article{okamoto2018optimal,
  title={Optimal covariance control for stochastic systems under chance constraints},
  author={Okamoto, Kazuhide and Goldshtein, Maxim and Tsiotras, Panagiotis},
  journal={IEEE Control Systems Letters},
  volume={2},
  number={2},
  pages={266--271},
  year={2018},
  publisher={IEEE}
}

@book{billingsley2013convergence,
  title={Convergence of probability measures},
  author={Billingsley, Patrick},
  year={2013},
  publisher={John Wiley \& Sons}
}

@article{hyvarinen2005estimation,
  title={Estimation of non-normalized statistical models by score matching.},
  author={Hyv{\"a}rinen, Aapo and Dayan, Peter},
  journal={Journal of Machine Learning Research},
  volume={6},
  number={4},
  year={2005}
}

@article{carrillo2019blob,
  title={A blob method for diffusion},
  author={Carrillo, Jos{\'e} Antonio and Craig, Katy and Patacchini, Francesco S},
  journal={Calculus of Variations and Partial Differential Equations},
  volume={58},
  pages={1--53},
  year={2019},
  publisher={Springer}
}

@book{pilipenko2014introduction,
  title={An introduction to stochastic differential equations with reflection},
  author={Pilipenko, Andrey},
  volume={1},
  year={2014},
  publisher={Universit{\"a}tsverlag Potsdam}
}

@book{JCL:2012,
  title={Robot motion planning},
  author={Latombe, Jean-Claude},
  volume={124},
  year={2012},
  publisher={Springer Science \& Business Media}
}

@book{evans2022partial,
  title={Partial differential equations},
  author={Evans, Lawrence C},
  volume={19},
  year={2022},
  publisher={American Mathematical Society}
}

@article{elamvazhuthi2023density,
  title={Density Stabilization Strategies for Nonholonomic Agents on Compact Manifolds},
  author={Elamvazhuthi, Karthik and Berman, Spring},
  journal={IEEE Transactions on Automatic Control},
 volume={69},
  number={3},
  pages={1448 - 1463},
  year={2024},
  publisher={IEEE}
}

@article{ho2020denoising,
  title={Denoising diffusion probabilistic models},
  author={Ho, Jonathan and Jain, Ajay and Abbeel, Pieter},
  journal={Advances in neural information processing systems},
  volume={33},
  pages={6840--6851},
  year={2020}
}

@article{anderson1982reverse,
  title={Reverse-time diffusion equation models},
  author={Anderson, Brian DO},
  journal={Stochastic Processes and their Applications},
  volume={12},
  number={3},
  pages={313--326},
  year={1982},
  publisher={Elsevier}
}

@article{figalli2010mass,
  title={Mass transportation on sub-Riemannian manifolds},
  author={Figalli, Alessio and Rifford, Ludovic},
  journal={Geometric and functional analysis},
  volume={20},
  pages={124--159},
  year={2010},
  publisher={Springer}
}

@article{elamvazhuthi2024benamou,
  title={Benamou-Brenier Formulation of Optimal Transport for Nonlinear Control Systems on Rd},
  author={Elamvazhuthi, Karthik},
  journal={arXiv preprint arXiv:2407.16088},
  year={2024}
}

@book{grune2017nonlinear,
  title={Nonlinear model predictive control},
  author={Gr{\"u}ne, Lars and Pannek, J{\"u}rgen and Gr{\"u}ne, Lars and Pannek, J{\"u}rgen},
  year={2017},
  publisher={Springer}
}

@inproceedings{freeman1996control,
  title={Control Lyapunov functions: New ideas from an old source},
  author={Freeman, Randy A and Primbs, James A},
  booktitle={Proceedings of 35th IEEE conference on decision and control},
  volume={4},
  pages={3926--3931},
  year={1996},
  organization={IEEE}
}

@article{krener1999feedback,
  title={Feedback linearization},
  author={Krener, Arthur J},
  journal={Mathematical control theory},
  pages={66--98},
  year={1999},
  publisher={Springer}
}

@inproceedings{caluya2020finite,
  title={Finite horizon density steering for multi-input state feedback linearizable systems},
  author={Caluya, Kenneth F and Halder, Abhishek},
  booktitle={2020 American Control Conference (ACC)},
  pages={3577--3582},
  year={2020},
  organization={IEEE}
}

@article{bakolas2018finite,
  title={Finite-horizon covariance control for discrete-time stochastic linear systems subject to input constraints},
  author={Bakolas, Efstathios},
  journal={Automatica},
  volume={91},
  pages={61--68},
  year={2018},
  publisher={Elsevier}
}

@inproceedings{song2020score,
  title={Score-Based Generative Modeling through Stochastic Differential Equations},
  author={Song, Yang and Sohl-Dickstein, Jascha and Kingma, Diederik P and Kumar, Abhishek and Ermon, Stefano and Poole, Ben},
  booktitle={International Conference on Learning Representations},
  year={2020}
}

@book{agrachev2019comprehensive,
  title={A comprehensive introduction to sub-Riemannian geometry},
  author={Agrachev, Andrei and Barilari, Davide and Boscain, Ugo},
  volume={181},
  year={2019},
  publisher={Cambridge University Press}
}

@book{bramanti2014invitation,
  title={An invitation to hypoelliptic operators and H{\"o}rmander's vector fields},
  author={Bramanti, Marco and others},
  volume={1298},
  year={2014},
  publisher={Springer}
}

\end{document}